\documentclass[11pt,reqno]{amsart}
\usepackage{amsthm}
\usepackage{amssymb}
\usepackage{latexsym}
\usepackage{multicol}
\usepackage{verbatim,enumerate}
\usepackage{accents}
\usepackage[usenames]{color}
\usepackage{float}
\usepackage[utf8]{inputenc}
\usepackage{hyperref}
\usepackage{amsmath, amscd}
\usepackage{soul}
\usepackage{tikz}
\usetikzlibrary{matrix,arrows,decorations.pathmorphing}
\usepackage{dirtytalk}

\advance\textwidth by 1.2in  \advance\oddsidemargin by -.6in \advance\evensidemargin by -.6in

\theoremstyle{plain}
\newtheorem*{prop}{Proposition}
\newtheorem{thm}{Theorem}
\newtheorem*{lem}{Lemma}
\newtheorem*{cor}{Corollary}

\theoremstyle{definition}
\newtheorem*{example}{Example}
\newtheorem*{defn}{Definition}

\newtheorem*{rem}{Remark}

\theoremstyle{remark}

\newcommand{\lie}[1]{\mathfrak{#1}}
\newcommand{\wh}[1]{\widehat{#1}}

\newcommand\bc{\mathbb C}

\newcommand\bz{\mathbb Z}
\newcommand\br{\mathbb R}

\def\a{\alpha}

\newenvironment{pf}{\proof}{\endproof}
\newcounter{cnt}
 \makeatletter
\def\mydggeometry{\makeatletter\dg@YGRID=1\dg@XGRID=20\unitlength=0.003pt\makeatother}
\makeatother \theoremstyle{remark}

\numberwithin{equation}{section}
\makeatletter
\def\section{\def\@secnumfont{\mdseries}\@startsection{section}{1}%
  \z@{.7\linespacing\@plus\linespacing}{.5\linespacing}%
  {\normalfont\scshape\centering}}
\def\subsection{\def\@secnumfont{\bfseries}\@startsection{subsection}{2}%
  {\parindent}{.5\linespacing\@plus.7\linespacing}{-.5em}%
  {\normalfont\bfseries}}
\makeatother

\begin{document}

\title[]{Borel-de Siebenthal theory for affine reflection systems}
\author{Deniz Kus}
\address{University of Bochum, Faculty of Mathematics, Universit{\"a}tsstr. 150, 44801 Bochum, 
Germany}
\email{deniz.kus@rub.de}
\thanks{}

\author{R. Venkatesh}
\address{Department of Mathematics, Indian Institute of Science, Bangalore 560012}
\email{rvenkat@iisc.ac.in}
\thanks{}

\subjclass[2010]{}
\begin{abstract}
We develop a Borel--de Siebenthal theory for affine reflection systems by classifying their maximal closed subroot systems. Affine reflection systems (introduced by Loos and Neher) provide a unifying framework for root systems of finite-dimensional semi-simple Lie algebras, affine and toroidal Lie algebras, and extended affine Lie algebras. In the special case of nullity $k$ toroidal Lie algebras, we obtain a one-to-one correspondence between maximal closed subroot systems with full gradient and triples $(q,(b_i),H)$, where $q$ is a prime number, $(b_i)$ is a $n$-tuple of integers in the interval $[0,q-1]$ and $H$ is a $(k\times k)$ Hermite normal form matrix with determinant $q$. This generalizes the $k=1$  result of Dyer and Lehrer in the setting of affine Lie algebras.
\end{abstract}

\maketitle

\section{Introduction}
Extended affine Lie algebras (EALA for short) appeared first in the paper \cite{KB90} under the name \say{irreducible quasi-simple Lie algebras} in which the authors proposed a system of axioms for an interesting class of Lie algebras. A more modern and standard reference for EALA are the papers \cite{AABGP97} and \cite{Neher04}, where the name \say{extended affine Lie algebras} appeared for the first time. EALA share many properties with familiar Lie algebras, like finite-dimensional simple Lie algebras or affine Kac-Moody algebras. Their structure theory is well--understood and is quite similar to the structure theory of affine Lie algebras. In fact, in the EALA setup the aforementioned familiar Lie algebras appear exactly as the nullity zero respectively nullity one EALA, where the rank of the group generated by the imaginary (isotropic) roots is called nullity. Motivated by the deformation theory of simply elliptic singularities, the first examples beyond the trivial cases were studied in a series of papers by Saito \cite{Sai85, Sai90}.

The motivation of this paper is to study regular subalgebras of EALA with more emphasis on the combinatorial aspects. The classification of such subalgebras is a fundamental problem arising in many mathematical contexts and were studied in the finite-dimensional case intensively by Borel, de Siebenthal and Dynkin. A Lie subalgebra $\lie g_0$ of a semisimple complex Lie algebra $\mathfrak{g}$ is called regular if it is invariant under a Cartan subalgebra $\lie h$. The root system of $\lie g_0$ with respect to $\lie h$ can be regarded naturally as a closed subroot system of that of $\lie g$ and straightforward calculations show that the problem of classifying the regular subalgebras of $\mathfrak{g}$ reduces to that of classifying the (maximal) closed subroot systems. Therefore we are faced with a purely combinatorial problem. The closed subroot systems up to isomorphism were determined by Borel and de Siebenthal \cite{BdS} (the maximal ones) and by Dynkin \cite[Ch. 2, §5]{Dynkin}. 

Similar questions can be asked for affine Kac--Moody algebras (equivalently nullity 1 EALA), but the situation is slightly more complicated. Nevertheless, the authors of \cite{FRT08} determined the possible types of regular subalgebras of an affine Kac-Moody algebra; the list was incomplete and has now been completed in \cite{RV17}. A slightly different approach to this problem has been made by Dyer and Lehrer in a series of papers \cite{DL11, DL11a}, where they classified the reflection subgroups of finite and affine Weyl groups. Classifying reflection subgroups amounts to classifying (real) subroot systems, but it is well-known (in the finite case) that the classification of all subroot systems may be deduced from that of the closed subroot systems \cite{Carter}. 

In this paper we ask the natural question for possible regular subalgebras of an extended affine Lie algebra; see \cite{CNPY16} for the conjugacy of Cartan subalgebras of EALA. From the combinatorial point of view the question can be formulated as follows. What are the maximal closed subroot systems of the root system of an EALA? It turns out that extended affine root systems (the roots of an EALA) form a special class of affine reflection systems, whose theory has been developed in \cite{LN11}. They provide a unifying framework for finite root systems, extended affine root systems and various generalizations thereof. In this paper, we answer the more general question and develop a Borel-de Siebenthal theory for affine reflection systems by classifying their maximal closed subroot systems explicitly (not only the type). In a sequel to this work \cite{BKV18}, we apply our results to obtain a complete list of regular subalgebras of EALA of nullity 2 listed in \cite{ABP14}. Let us first discuss the difficulties passing beyond nullity 1. 

The important fact which was used by Dyer and Lehrer to characterize subroot systems of affine root systems is that any pointed reflection subspace $A\subseteq \mathbb{Z}$ (i.e. $A-2A\subseteq A$ and $0\in A$) is a subgroup of $\mathbb{Z}$ (see \cite[Lemma 22]{DL11}) and hence has the form $r\bz$.
Of course, this fact is no longer true for arbitrary abelian groups and not even for $\mathbb{Z}\oplus \mathbb{Z}$. The latter abelian group appears for example as the group generated by the imaginary roots of a nullity 2 untwisted toroidal Lie algebra. Affine reflection systems are constructed from a finite root system $\mathring{\Phi}$ and an extension datum, where an extension datum is a collection of sets $\Lambda_{\alpha},\  \alpha\in \mathring{\Phi}\cup\{0\}$ satisfying certain conditions (see Section \ref{section2} for a precise definition). Since we allow arbitrary sets, it is more difficult to deal with affine reflection systems. Nevertheless, the additional properties give us at least some restrictions and relations among the sets $\Lambda_\a$; see Lemma \ref{inclusionen} for a complete list of constraints on $\Lambda_\a$. 

We will now discuss our results in the special case of an untwisted toroidal Lie algebra; all details can be found in Section~\ref{section8}. Toroidal Lie algebras are $k$--variable generalizations of affine Kac--Moody algebras. Their root systems have the form $\Phi=\{\alpha\oplus \mathbb{Z}^k : \alpha\in \mathring{\Phi}\},$ where $\mathring{\Phi}$ is a finite root system and we will assume for simplicity that $\mathring{\Phi}$ is irreducible, reduced and of rank $n\geq 2$. Let $\Psi\subseteq \Phi$ a maximal closed subroot system. Our main result in this case implies the following two cases. Either there is a maximal closed subroot system $\mathring{\Phi}'$ in $\mathring{\Phi}$ à la Borel-de Siebenthal such that
$\Psi=\{\alpha \oplus \bz^k: \alpha\in\mathring{\Phi}'\}$
or there exists a prime number $q$ and a Hermite normal form matrix $U^{\top}=(u_1,\dots,u_k)$ with $\text{det}(U)=q$ and a tuple of integers $(b_{\a_i})_{}\in [0,q-1]^{n}$ (for each simple root) such that
$$\Psi=\{\alpha \oplus b_{\a} e_\ell+\bz u_1+\cdots +\bz u_k: \alpha\in\mathring{\Phi}\},$$
where $\ell$ is the unique integer such that $u_\ell=q e_\ell$ ($e_\ell$ is the $\ell$--th unit vector). Therefore, in this particular case, the maximal closed subroot systems (with full gradient) are in one-to-one correspondence with triples $(q,(b_{\a_i})_{},U)$, where $q$ is a prime number, $(b_{\a_i})_{}\in [0,q-1]^{n}$ and $U$ is a Hermite normal form matrix with determinant $q$. This generalizes the $k=1$ result of Dyer and Lehrer \cite{DL11a}; see also \cite{RV17}. The twisted case is quite more challenging and we discuss an example of Saito's root system of a twisted toroidal Lie algebra in Section~\ref{section8}.

\textit{Organization of the paper:} In Section~\ref{section2} we introduce affine reflection systems and establish their basic properties. Furthermore, we discuss several examples. In Section~\ref{section3} we define subroot systems of affine reflection systems and prove several properties of their gradient. The classification of the maximal closed subroot systems is worked out in Section~\ref{section4}--\ref{section7} and in Section~\ref{section8} we apply our results to the setting of toroidal Lie algebras. 

\textit{Organization of the classification:} Let $\Phi$ an irreducible affine reflection system and $\Psi$ a maximal closed subroot system. Then the gradient $\text{Gr}(\Psi)$ (see Section~\ref{31}) can be semi--closed or closed (see Section \ref{33}) and in the closed case it can be properly contained in $\mathring{\Phi}$ or it can be equal to $\mathring{\Phi}$. We also distinguish between $\mathring{\Phi}$ reduced and non--reduced. The following table summarizes the organization of the classification.

\begin{equation*} \label{n:table2} {\renewcommand{\arraystretch}{1.2}
\begin{array}{|c| c |c|} \hline

 (\Psi, \Phi) &  \mathring{\Phi} \text{ reduced } &  \mathring{\Phi} \text{ non--reduced }\\  \hline 

\rule{0pt}{15pt}  \text{Gr}(\Psi) \text{ semi--closed }&  \text{Section  \ref{section4}}: \ \text{Theorem \ref{Cng2case}}, \text{Theorem \ref{F4case}} \text{ and } \text{Theorem \ref{Bncase}} & \\ \cline{1-2} 

\rule{0pt}{15pt}\text{Gr}(\Psi)=\mathring{\Phi}   & \text{Section \ref{section5}}:\  \text{Theorem  \ref{part2g2f4}} \text{ and } \text{Theorem \ref{Bnformcasefull}} &  \text{Section \ref{section7}}: \\ \cline{1-2}

\rule{0pt}{15pt}\text{Gr}(\Psi) \text{ proper closed } &  \text{Section \ref{section6}}: \ \text{Theorem \ref{propgrm}}& \text{Theorem \ref{mainBCn}}  \\ \cline{1-3}

\end{array}}
\end{equation*}
\section{Extension datum and affine reflection system}\label{section2}
\subsection{}\label{section21}Let $V$ a fixed Eucledian real vector space, i.e. $V$ is endowed with a positive definite symmetric bilinear form  $(\cdot,\cdot)$. For the rest of this paper we denote by $\mathring{\Phi}$ a root system in $V$, i.e. $ \mathring{\Phi}$ is a finite subset of $V$ satisfying the following properties (see \cite[Chapter VI]{Bou46} or \cite[Section 9.2]{Hu80}):
$$0\notin \mathring{\Phi},\ \ \mathring{\Phi} \text{ spans $V$},\ \ s_{\alpha}(\mathring{\Phi})=\mathring{\Phi},\ \forall\alpha\in\mathring{\Phi},\ \ (\beta,\alpha^{\vee})\in\bz,\ \forall \alpha,\beta\in \mathring{\Phi},$$
where $\alpha^{\vee}:=2\alpha/(\alpha,\alpha)$ and $s_{\alpha}$ is the reflection of $V$ defined by $s_{\alpha}(v)=v-(v,\alpha^{\vee})\alpha$, $v\in V$.
We call $\mathring{\Phi}$ reduced if it satisfies the additional property $\br\alpha\cap \mathring{\Phi}=\{\pm \alpha\}$ for $\alpha\in \mathring{\Phi}$. Moreover, we call $\mathring{\Phi}$ irreducible (or connected) whenever $\mathring{\Phi}=\mathring{\Phi}'\cup \mathring{\Phi}''$ with $(\mathring{\Phi}',\mathring{\Phi}'')=0$ implies $\mathring{\Phi}'=\emptyset$ or $\mathring{\Phi}''=\emptyset$. Any root system can be written as a direct sum of irreducible root systems (see \cite[Proposition 6]{Bou46}) and the reduced irreducible root systems were classified in terms of their Dynkin diagrams (see \cite[Theorem 11.4]{Hu80}). There are the classical types $A_n, n\geq 1, B_n, n\geq 2, C_n, n\geq 3, D_n, n\geq 4$ and the exceptional types $E_{6,7,8}, F_4$ and $G_2$. For a direct contruction of these root systems we refer to \cite[Section 12.1]{Hu80}. Moreover, there is only one non--reduced irreducible root system of rank $n$, namely 
\begin{equation}\label{nonre}BC_n=B_n\cup C_n=\{\pm\epsilon_i: 1\leq i\leq n\}\cup \{\pm\epsilon_i\pm \epsilon_j: 1\leq i,j\leq n\}\backslash\{0\},\end{equation}
where $\epsilon_1,\dots,\epsilon_n$ denotes an orthonormal basis of $V$ with respect to $(\cdot,\cdot)$.
Reduced root systems appear in the context of finite--dimensional semi--simple Lie algebras as the set of roots of the Lie algebra with respect to a Cartan subalgebra \cite{Hu80} and non--reduced root systems appear in the context of infinite--dimensional Lie algebras \cite{K90}; see also Example~\ref{ex1}.
\subsection{}
If $\mathring{\Phi}$ is a reduced irreducible root system in $V$, then at most two root lengths occur and all roots of a given length are conjugate under the Weyl group $W(\mathring{\Phi})$, which by definition is the group generated by $s_{\alpha},\alpha\in\mathring{\Phi}$ (see for example \cite[Section 10.4]{Hu80}). We denote the set of short roots (resp. long roots) by $\mathring{\Phi}_s$ (resp. $\mathring{\Phi}_\ell$) and if there is only one root length we say that every root is long. If $\mathring{\Phi}$ is non--reduced irreducible we define (compare with \eqref{nonre}):
$$\mathring{\Phi}_s=\{\pm\epsilon_i,\ 1\leq i\leq n\},\ \ \mathring{\Phi}_\ell=\{\pm\epsilon_i\pm\epsilon_j: 1\leq i\neq j\leq n\},\ \ \mathring{\Phi}_d=\{\pm2\epsilon_i,\ 1\leq i\leq n\}.$$
Note that $\mathring{\Phi}_d$ is simply the set of divisible roots, i.e. $\mathring{\Phi}_d=\{\alpha\in \mathring{\Phi}: \alpha/2\in \mathring{\Phi}\}$ and define the non--divisible roots by $\mathring{\Phi}_{nd}=\mathring{\Phi}_{}\backslash \mathring{\Phi}_{d}$. Further, set
$$m_{\mathring{\Phi}}=\begin{cases}1,& \text{if $\mathring{\Phi}$ is of type $A_n,D_n,E_{6,7,8}$,}\\
2,& \text{if $\mathring{\Phi}$ is of type $B_n, C_n, F_4$,}\\
3,& \text{if $\mathring{\Phi}$ of type $G_2$,}\\
4,& \text{if $\mathring{\Phi}$ is non--reduced irreducible.}
\end{cases}
$$
The following lemma is an easy exercise.
\begin{lem}\label{hum}
Let $\mathring{\Phi}$ an irreducible root system not of simply--laced type. Then the following holds:
\begin{enumerate}
\item Let $\alpha\in \mathring{\Phi}_x,\ \beta\in \mathring{\Phi}_y$ such that $\alpha+\beta\in \mathring{\Phi}_z$, where $x,y,z\in\{s,\ell,d\}$. Then we have
\begin{align*}\hspace{1,2cm}&(x,y,z)\in\{(s,s,\ell), (s,\ell,s),(\ell,\ell,\ell)^*\},\ \text{if $\mathring{\Phi}$ is of type $B_n$,}&\\&
(x,y,z)\in\{(s,s,\ell), (s,\ell,s),(s,s,s)\},\ \text{if $\mathring{\Phi}$ is of type $C_n$,}&\\&
(x,y,z)\in\{(s,s,\ell), (s,\ell,s),(s,s,s),(\ell,\ell,\ell)\},\ \text{if $\mathring{\Phi}$ is of type $F_4$ or $G_2$,}&\\&
(x,y,z)\in\{(s,s,\ell),(s,\ell,s), (\ell,\ell,\ell)^*,(\ell,\ell,d), (\ell,d,\ell), (s,d,s),(s,s,d)\},\ \text{if $\mathring{\Phi}$ is of type $BC_n$,}\end{align*}
where the superscript $*$ means that it only appears when $n\geq 3$.
\item If $\mathring{\Phi}$ is not of type $BC_1$, there exists $\alpha,\beta\in\mathring{\Phi}_s$ such that $\alpha+\beta\in \mathring{\Phi}_\ell$. If $\mathring{\Phi}$ is of type $BC_n$ we can find $\alpha,\beta\in\mathring{\Phi}_s$ such that $\alpha+\beta\in\mathring{\Phi}_d$.
\end{enumerate}\hfill\qed
\end{lem}

\subsection{}Now we define our main object of study, namely affine reflection systems whose theory is developed by Loos and Neher in \cite{LN11} (see also \cite[Section 3.1]{N11F}). They generalize the notion of affine root systems (the set of roots of an affine Kac--Moody algebra) or more generally the notion of extended affine root systems (the set of roots of an extended affine Lie algebra). Our goal will be to work out a Borel-de Siebenthal theory for affine reflection systems and classify their maximal closed subroot systems; see Section\ref{31} for a precise definition. 

Let $X$ a finite--dimensional real vector space, $(\cdot,\cdot)_X$ a symmetric bilinear form on $X$ and $\Phi\subseteq X$ a subset. Define 
\begin{align*}X^0&=\{x\in X: (x,X)_X=0\},\\ \Phi^{\text{im}}&=\{\alpha\in \Phi: (\alpha,\alpha)_X=0\} ,\ (\text{imaginary roots})\\ \Phi^{\text{re}}& =\{\alpha\in \Phi: (\alpha,\alpha)_X\neq 0\},\ (\text{real roots}).\end{align*}
\begin{defn}\label{affrefl}
We call the triple $(X,\Phi,(\cdot,\cdot)_X)$ an affine reflection system if the following axioms are satisfied:
\begin{enumerate} 
\item [(i)] $0\in \Phi,\  \Phi \text{ spans $X$},\  \Phi^{\text{im}}\subseteq X^0$,
\item  [(ii)] $s_{\alpha}(\Phi)=\Phi,\ \ \forall \alpha\in \Phi^{\text{re}},$
\item  [(iii)] $(\Phi,\alpha^{\vee})_X\subseteq\bz \text{ is a finite subset $\forall \alpha\in \Phi^{\text{re}}$,}$
\end{enumerate}
where $s_{\alpha}$ and $\alpha^{\vee}$ are defined as in Section~\ref{section21} with respect to the form $(\cdot,\cdot)_X$. The rank of the torsion free abelian group $\bz[\Phi^{\text{im}}]$ is called the nullity of $(X,\Phi,(\cdot,\cdot)_X)$. Similarly, as in Section~\ref{section21}, we can define irreducible and reduced affine reflection systems. 

An isomorphism from an affine reflection system $(X,\Phi,(\cdot,\cdot)_X)$ to another affine reflection system $(X',\Phi',(\cdot,\cdot)_{X'})$ is a vector space isomorphism $f:X\rightarrow X'$ such that $f(X^0)=(X^0)'$ and $f(\Phi^{\text{re}})=(\Phi^{\text{re}})'$.
\end{defn}

\begin{rem}The requirement $0\in \Phi$ follows the approach of Loos and Neher (see for example \cite[Section 3.1]{N11F}) and contradicts the traditional approach to root systems in which 0 is not considered as a root (see \cite{Bou46} or \cite{Hu80}). Because of the notational advantages, we decided to include $0$ and follow their approach.
\end{rem}
Before we discuss concrete examples, let us state the structure theorem for affine reflection systems; see Theorem \ref{strthm}.
\subsection{}\label{sconsaff}It turns out that any reflection system can be constructed from a (finite) root system and an extension datum, which we will define now. 
\begin{defn}\label{extda}Let $\mathring{\Phi}$ a (finite) root system and $Y$ a finite--dimensional real vector space. We call a collection $(\Lambda_{\alpha}: \alpha\in \mathring{\Phi}\cup\{0\})$, $\Lambda_{\alpha}\subseteq Y$ an extension datum of type $(\mathring{\Phi},Y)$ if it satisfies the following axioms:
\begin{enumerate}
\item $\Lambda_{\beta}-(\beta,\a^{\vee})\Lambda_{\a}\subseteq \Lambda_{s_{\a}(\beta)},\ \ \forall \alpha,\beta\in \mathring{\Phi}$.
\item $0\in\Lambda_{\alpha}$ for $\alpha\in \mathring{\Phi}_{nd}\cup\{0\}$ and $\Lambda_{\alpha}\neq \emptyset$ for $\alpha\in \mathring{\Phi}_{d}$.
\item $Y=\text{span}_{\mathbb{R}}(\bigcup_{\alpha\in \mathring{\Phi}\cup\{0\}} \Lambda_\alpha)$.
\end{enumerate}
\end{defn}
Now given an extension datum of type $(\mathring{\Phi},Y)$ we can construct an affine reflection system as follows. Recall that $(\cdot,\cdot)$ is the fixed positive definite symmetric bilinear form on $V$. Define 
\begin{equation}\label{consaff}X=V\oplus Y,\ \ \Phi=\bigcup_{\alpha\in \mathring{\Phi}\cup\{0\}}\alpha\oplus \Lambda_{\alpha},\ \ (v_1\oplus y_1, v_2\oplus y_2)_X=(v_1,v_2).\end{equation}
The following result is the structure theorem for affine reflection systems \cite[Theorem 4.6]{LN11}.
\begin{thm}\label{strthm}
Let $\mathring{\Phi}$ a (finite) root system and $(\Lambda_{\alpha}: \alpha\in \mathring{\Phi}\cup\{0\})$ an extension datum of type $(\mathring{\Phi},Y)$. The triple $(\Phi,X,(\cdot,\cdot)_X)$ constructed in \eqref{consaff} is an affine reflection system with 
$$\Phi^{\text{im}}=\Lambda_0,\ \ X^0=Y,\ \ \Phi^{\text{re}}=\bigcup_{\alpha\in \mathring{\Phi}}\alpha\oplus \Lambda_{\alpha}.$$
Moreover, any affine reflection system is isomorphic to an affine reflection system  constructed in this way and $(\Phi,X,(\cdot,\cdot)_X)$ is irreducible if and only if the underlying finite root rystem $\mathring{\Phi}$ is irreducible.
\hfill\qed
\end{thm}
\begin{example}\label{ex1}
\begin{enumerate}
\item [(i)] Let $\mathring{\Phi}$ a (finite) root system, then $\mathring{\Phi}\cup\{0\}$  is an affine reflection system of nullity $0$. We choose $Y=0$ in Definition~\ref{extda}.
\item [(ii)] Let $\lie g$ a finite-dimensional simple complex Lie algebra and $\sigma$ a diagram automorphism with respect to a Cartan subalgebra $\lie h$, which is of order $m\in\{1,2,3\}$. Let $\xi$ be a primitive $m$--th root of unity. We have
$$\lie g=\bigoplus_{j\in \bz/m\bz} \lie g_j,\ \  \lie g_{j}=\{x\in \lie g: \sigma(x)=\xi^j x\}.$$
It is known that $\lie g_0$ is again a finite-dimensional complex simple Lie algebra with Cartan subalgebra $\lie h_0=\lie h\cap \lie g_0$. Moreover, $\lie g_j$ is a $\lie g_0$--module and we denote the set of non--zero weights of $\lie g_j$ with respect to $\lie h_0$ by $\mathring{\Phi}_j$. We have that 
$$\mathring{\Phi}:=\mathring{\Phi}_0 \cup\cdots\cup \mathring{\Phi}_{m-1}\subseteq \lie h_0^*$$
is again a (finite) root system. The type of $\mathring{\Phi}$ can be extracted from \cite[Section 7.8, 7.9, 8.3]{K90} and is summarized in the following table
\begin{equation*} \label{n:table1} {\renewcommand{\arraystretch}{1.5}
\begin{array}{|c| c| c| c |c|} \hline
 (\lie g, m) &  \mathring{\Phi}_{0} & \mathring{\Phi}_{1} & \mathring{\Phi}_{2} &\mathring{\Phi} \\ \hline 
  (A_{2n}, 2) &  B_n 
    &  \mathring{\Phi}_0 \cup \{\pm 2\epsilon_i : 1 \le i \le n \} & / & BC_n \\ \hline
 (A_{2n-1}, 2)   & C_n & (C_{n})_ {s} & / &C_n \\ \hline
(D_{n+1}, 2) &  B_n & (B_{n})_ {s} & /& B_n \\
\hline
 (E_6,2) & F_4 & (F_{4})_ {s}  &/ & F_4 \\ \hline
 (D_4, 3) &  G_2 & (G_{2})_ {s}  & (G_{2})_ {s} & G_2 \\ \hline
\end{array}}
\end{equation*}

The corresponding affine Kac--Moody algebra $\widehat{\mathcal{L}}(\lie g,\sigma)$ is defined by
$$\widehat{\mathcal{L}}(\lie g,\sigma)=\mathcal{L}(\lie g,\sigma)\oplus \bc c\oplus \bc d,\ \ \mathcal{L}(\lie g,\sigma)=\bigoplus_{j\in \bz/m\bz} \lie g_{j}\otimes \bc[t^{\pm m}]t^{j},$$ where $\mathcal{L}(\lie g,\sigma)$ is called the loop algebra, $\mathcal{L}(\lie g,\sigma)\oplus \bc c$ is the universal central extension of the loop algebra and $d=t\partial_t$ is the degree derivation. For more details we refer the reader to \cite[Section 7,8]{K90}. The set of roots of $\widehat{\mathcal{L}}(\lie g,\sigma)$ with respect to the Cartan subalgebra $\lie h_0\oplus \bc c\oplus \bc d$ is exactly $\Phi\backslash\{0\}$, where
$$\Phi:=\{\alpha\oplus r\delta: \alpha\in \mathring{\Phi}_j\cup\{0\},\ \ r\in j+\bz m,\ 0\leq j< m\}.$$
We claim that $\Phi$ is an affine reflection system of nullity 1. Let $V:=(\lie h_0)^*_{\mathbb{R}}$ the real span of $\mathring{\Phi}$ and let $(\cdot,\cdot)$ the dual of the Killing form $\kappa(\cdot,\cdot)$ restricted to $(\lie h_0)_{\mathbb{R}}\times (\lie h_0)_{\mathbb{R}}$ (the restriction is still non--degenerate). Set
$$X=V\oplus \br \delta=V\oplus Y,\ \ (\beta_1\oplus r_1\delta,\beta_2\oplus r_2\delta)_X=(\beta_1,\beta_2).$$
Then, we have $$X^{0}=\br \delta,\ \ \Phi^{\text{re}}=\{\alpha\oplus r\delta\in \Phi : \alpha\neq 0\},\ \Phi^{\text{im}}=\bz\delta.$$
If $m=1$, we set  $\Lambda_\alpha=\bz\delta$ for all $\alpha\in \mathring{\Phi}\cup\{0\}$. Otherwise, set
$$\Lambda_{\alpha}=\begin{cases} \bz\delta,& \text{if $\alpha\in \mathring{\Phi}_s\cup\{0\}$}\\
\bz\delta,& \text{if $\alpha\in \mathring{\Phi}_\ell$,\ $\mathring{\lie g}=A_{2n}$, $m=2$}\\
2\bz\delta,& \text{if $\alpha\in \mathring{\Phi}_\ell$,\ $\mathring{\lie g}\neq A_{2n}$, $m=2$}\\
3\bz\delta,& \text{if $\alpha\in \mathring{\Phi}_\ell$,\ $m=3$}\\
(2\bz+1)\delta,& \text{if $\alpha\in \mathring{\Phi}_d$}\end{cases}
$$
It is immediate to show that $\mathcal{R}=(\Lambda_{\alpha}, \alpha\in \mathring{\Phi}\cup\{0\})$ is an extension datum (see Definition \ref{extda}) and $\Phi$ is the affine reflection system constructed from $\mathcal{R}$ as in \eqref{consaff}. Since $\bz[\Phi^{\text{im}}]\cong \bz$, the nullity of $\Phi$ is 1.
\item  [(iii)] Toroidal Lie algebras underlie a similar construction as affine Kac--Moody algebras by replacing the loop algebra by the multiloop algebra $\mathcal{L}(\lie g)=\lie g\otimes \bc[t_1^{\pm},\dots,t_k^{\pm}]$ and $\sigma$ by a family of finite order automorphisms $(\sigma_1,\dots,\sigma_k)$. For simplicity we dicuss the untwisted case only, i.e. when all authomorphisms are trivial. We follow the definition of toroidal Lie algebras used for example in \cite{N11F} or \cite[Section 1]{Ra04} ($\tilde \tau$ in his notation). Let 
$$C=\bigoplus_{p\in \bz^k} C_p, \ \ C_p:=\bc^k/\bc p,$$ and define 
$$\psi: \mathcal{L}(\lie g)\times \mathcal{L}(\lie g)\rightarrow C,\ \ \psi(x\otimes t^{r},y\otimes t^s)=\bigoplus_{p\in\bz}\kappa(x,y)\delta_{r+s,-p}\overline r,$$
where $t^r:=t_1^{r_1}\cdots t_k^{r_k}$ for $r=(r_1,\dots,r_k)\in\bz^k$. It follows from \cite[Theorem 1.5]{N11F} that $\mathcal{L}(\lie g)\oplus C$ is the universal central extension of the multiloop algebra whose Lie bracket is given by $$[x\otimes t^r+\mu c,y\otimes t^s+\lambda c]=[x,y]\otimes t^{r+s}+\psi(x\otimes t^{r},y\otimes t^s).$$ The toroidal Lie algebra is then given by
$$\Gamma(\lie g)=\mathcal{L}(\lie g)\oplus C\oplus D,$$
where $D=\text{span}_{\bc}\{t_i\partial_{t_i}: 1\leq i\leq k\}$ is the space of degree derivations. Denote by $\mathring{\Phi}$ the root system of $\lie g$. The set of roots of the toroidal Lie algebra $\Gamma(\lie g)$ with respect to the Cartan subalgebra $\lie h\oplus C_0\oplus D$ is $\Phi\backslash \{0\}$, where
$$\Phi=\{\alpha\oplus\delta_r: \alpha\in \mathring{\Phi}\cup \{0\},\ r\in \bz^k\},$$
where $\delta_r:=r_1\delta_1+\cdots+r_k\delta_k$. We claim that $\Phi$ is an affine reflection system of nullity $k$. Again we denote by $\lie h_{\mathbb{R}}$ the real span of $\mathring{\Phi}$ and $(\cdot,\cdot)$ the dual of the Killing form. Set
$$X=\lie h_{\mathbb{R}}^*\oplus (\br \delta_1\oplus\cdots \oplus \br \delta_k)=V\oplus Y,\ \ (\beta_1\oplus \delta_{r_1},\beta_2\oplus \delta_{r_2})_X=(\beta_1,\beta_2).$$
Then, we have $$X^{0}=\br\delta_1\oplus\cdots \oplus \br \delta_k,\ \ \Phi^{\text{re}}=\{\alpha\oplus \delta_r: \alpha\in \mathring{\Phi},\ r\in\bz^k\},\ \Phi^{\text{im}}=\bz \delta_1\oplus\cdots \oplus \bz \delta_k.$$
Now it is easy to see that $\Lambda_{\alpha}=\bz^k$ for all $\alpha\in \mathring{\Phi}\cup \{0\}$ defines an extension datum and $\Phi$ is constucted as in $\eqref{consaff}$. Since $\Phi^{\text{im}}\cong \bz^k$, the nullity of $\Phi$ is $k$.
\end{enumerate}
\end{example}
The above examples show that the set of roots of toroidal Lie algebras (in particular of affine Kac--Moody algebras) are affine reflection systems. More generally, the roots of an extended affine Lie algebra form also an affine reflection system; see \cite[Section 3.4]{N11F} for more details. So the theory of affine reflection systems provide a common framework for (finite) root systems, affine root systems and extended affine root systems. 
\subsection{}We finish this section by some elementary properties of an extension datum. Most of the results are stated in \cite[Exercise 3.16]{N11F} and \cite[Theorem 3.18]{N11F}. We give some of the proofs for the readers convenience.
\begin{prop}\label{prop1a}
Let $(\Lambda_{\alpha},\alpha\in\mathring{\Phi}\cup \{0\})$ an extension datum of type $(\mathring{\Phi},Y)$. Then, for all $\alpha\in\mathring{\Phi}$ we have
$$\Lambda_{\alpha}=\Lambda_{w(\alpha)},\ \forall w\in W(\mathring{\Phi}),\ \ \Lambda_{\alpha}=\Lambda_{-\alpha}=-\Lambda_{\alpha}.$$
\begin{proof}
Let $\alpha\in\mathring{\Phi}$ and $\beta\in \mathring{\Phi}_{nd}$. Then we have by Definition~\ref{extda}(1):
$$\Lambda_{\alpha}-(\alpha,\beta^{\vee})\Lambda_{\beta}\subseteq \Lambda_{s_{\beta}(\alpha)}.$$
Since $0\in \Lambda_{\beta}$ we get $\Lambda_{\alpha}\subseteq \Lambda_{s_{\beta}(\alpha)}$. Continuing in this way we get $\Lambda_{\alpha}\subseteq \Lambda_{w(\alpha)}$ for all $w\in W(\mathring{\Phi})$; note that the Weyl group is generated by $s_{\beta},\beta\in \mathring{\Phi}_{nd}$. Since $\Lambda_{w(\alpha)}\subseteq \Lambda_{w^{-1}w(\alpha)}$, the first part of the claim follows. For the second part we consider $\Lambda_{\alpha}-2\Lambda_{\alpha}\subseteq \Lambda_{-\alpha}.$ This gives $-\Lambda_{\alpha}=\Lambda_{-\alpha}$ for all $\alpha\in \mathring{\Phi}$ and $\Lambda_{\alpha}=\Lambda_{-\alpha}$ for all $\alpha\in \mathring{\Phi}_{nd}$, since $0\in \Lambda_{\alpha}$ whenever $\alpha$ is non--divisible. It remains to show $\Lambda_{\alpha}=\Lambda_{-\alpha}$ for all $\alpha\in \mathring{\Phi}_{d}$. But this follows since $\alpha/2\in \mathring{\Phi}_{nd}$ and $\Lambda_{\alpha}-(\alpha,(\alpha/2)^{\vee})\Lambda_{\alpha/2}\subseteq \Lambda_{-\alpha}$.
\end{proof}
\end{prop}
In particular, if $\mathring{\Phi}$ is irreducible we know that the Weyl group acts transitively on the roots of the same length and hence the above proposition implies that an extension datum of type $(\mathring{\Phi},Y)$ consists of at most 4 different subsets, namely $\Lambda_{0},\Lambda_{s},\Lambda_{\ell},\Lambda_{d}$ defined in the obvious way. For example, in the simply--laced case we have only one subset and by convention $\Lambda_s=\Lambda_\ell$. Hence if $(X,\Phi,(\cdot,\cdot)_X)$ is an irreducible affine reflection system, we know from Theorem~\ref{strthm} that there exists a (finite) irreducible root system $\mathring{\Phi}$ and an extension datum $(\Lambda_{0},\Lambda_{s},\Lambda_{\ell},\Lambda_{d})$ of type $(\mathring{\Phi},Y)$ such that (up to isomorphism)
\begin{equation*}\Phi=(0\oplus \Lambda_0) \cup \bigcup_{\alpha\in \mathring{\Phi}_s} (\alpha\oplus \Lambda_s) \cup  \bigcup_{\alpha\in \mathring{\Phi}_\ell} (\alpha\oplus \Lambda_\ell)\cup  \bigcup_{\alpha\in \mathring{\Phi}_d} (\alpha\oplus \Lambda_d)\end{equation*}
It is obvious from Definition~\ref{extda}(1) that there are certain relations among the subsets $\Lambda_s,\Lambda_{\ell},\Lambda_d$. We summarize the relations in the next lemma (see \cite[Theorem 3.18]{N11F}).
\begin{lem}\label{inclusionen} Let $\mathring{\Phi}$ an irreducible finite root system and $(\Lambda_{\alpha}, \alpha\in \mathring{\Phi}\cup\{0\})$ an extension datum of type $(\mathring{\Phi},Y)$. Then we have $\Lambda_x+2\Lambda_x=\Lambda_x$ for $x\in\{s,\ell,d\}$. Moreover,
\begin{itemize}

\item[\rm (i)] If $\mathring{\Phi}$ is simply--laced not of type $A_1$, then  $\Lambda_\ell$ is a subgroup.

\item[\rm (ii)] If $\mathring{\Phi}$ is of type $B_n$, $C_n$ or $F_4$, then
$$\Lambda_\ell + 2 \Lambda_s \subseteq \Lambda_\ell
\quad \hbox{and}\quad \Lambda_s +  \Lambda_\ell \subseteq \Lambda_s.$$ Moreover,

    \begin{itemize}
        \item[$\bullet$] $\Lambda_\ell $ is a subgroup if $\mathring{\Phi}$ is of type $B_n, n\geq 3$ or $F_4$, and

        \item[$\bullet$] $\Lambda_s$ is a subgroup if $\mathring{\Phi}$ is of type $C_n$ or $F_4$.
    \end{itemize}

\item[\rm(iii)]  If $\mathring{\Phi}$ is of type $G_2$, then $\Lambda_s$ and $\Lambda_\ell$ are subgroups satisfying
$$ \Lambda_\ell + 3 \Lambda_s  \subseteq \Lambda_\ell
\quad \hbox{and}\quad \Lambda_s + \Lambda_\ell \subseteq \Lambda_s.$$


\item[\rm (iv)]  If $\mathring{\Phi}$ is of type $BC_n, n\geq 2$, then 
\begin{center} \begin{tabular}{ccc}
 $\Lambda_\ell + 2 \Lambda_s \subseteq \Lambda_\ell$, &\quad &
  $\Lambda_s + \Lambda_\ell \subseteq \Lambda_s,$ \\
 $\Lambda_d + 2 \Lambda_\ell \subseteq \Lambda_d$, &\quad &
     $ \Lambda_\ell +\Lambda_d \subseteq \Lambda_\ell$, \\
  $\Lambda_d + 4 \Lambda_s \subseteq \Lambda_d$, &\quad &
 $\Lambda_s + \Lambda_d \subseteq \Lambda_s.$
\end{tabular} \end{center}
In addition, if $n\ge 3$ then $\Lambda_\ell$ is a subgroup.
\end{itemize}
\hfill\qed
\end{lem}
For example, the fact that $\Lambda_{\ell}$ (resp. $\Lambda_s$) is a subgroup in type $B_n, n\geq 3$ (resp. $C_n, n\geq 3$) can be easily seen from Definition~\ref{extda}(1) by choosing two arbitrary simple long roots (resp. simple short roots) $\alpha$ and $\beta$ connected in the Dynkin diagram by a single line.
\begin{defn}We call an extension datum untwisted if $\Lambda_{\alpha}=\Lambda_{\beta}$ for all $\alpha,\beta\in\mathring{\Phi}$ and twisted otherwise. If $\mathring{\Phi}$ is irreducible, the definition of untwisted is equivalent to $\Lambda_s\subseteq \Lambda_d\subseteq \Lambda_{\ell}$ by Lemma~\ref{inclusionen}. We call an affine reflection system (un)twisted, if the corresponding extension datum is (un)twisted.
\end{defn}
\vspace{0,1cm}

\textbf{Mild Assumption:} \textit{For the rest of the paper we will assume that $\Lambda_\ell$ is also a subgroup in the case when $\mathring{\Phi}$ is of type $B_2$ or $BC_2$. For simplicity, we also want to exclude the situation when $\mathring{\Phi}$ is of rank one in the rest of the paper (without further comment).} Nevertheless, a similar classification can be done for types $A_1$ and $BC_1$, but involves many technical calculations.

\section{Subroot systems of affine reflection systems}\label{section3}Given an affine reflection system $\Phi$, it is natural to ask what are its subroot systems. In this section we give the necessary definitions and elementary properties of subroot systems. 
\subsection{}\label{31}Let $\Phi$ an affine reflection system. A proper non--empty subset $\Psi\subseteq \Phi^{\text{re}}$ of the real roots is called
\begin{enumerate}
\item a subroot system, if we have $s_{\alpha}(\beta)\in \Psi$ for all $\alpha,\beta\in\Psi$;
\item closed, if $\alpha,\beta\in \Psi$ and $\alpha+\beta\in \Phi^{\text{re}}$ implies $\alpha+\beta\in \Psi$;
\item maximal closed subroot system, if $\Psi$ is a closed subroot system and $\Psi\subseteq \Psi'\subsetneq \Phi^{\text{re}}$ implies $\Psi=\Psi'$ for all closed subroot systems $\Psi'$.
\end{enumerate}
\begin{rem}\label{automatic}Let $\Psi$ a closed subset of an affine reflection system $\Phi$ such that $\Psi=-\Psi$ and $s_{\alpha}(\beta)\in\Psi$ for all $\alpha,\beta\in \Psi$ with $\beta\pm\alpha\in\Phi^{\text{im}}$ or $\beta\pm 2\a\in\Phi^{\text{im}}$. Since all root strings in $\Phi$ are unbroken \cite[Corollary 5.2]{LN11} we automatically have that $\Psi$ is a subroot system. To see this, let $\alpha,\beta\in \Psi$ such that $(\beta,\alpha^{\vee})_X\in\bz_+$. If $\beta-s\alpha\in \Phi^{\text{im}}$ for some $s\in\mathbb{Z}_+$ we must have $s\in\{1,2\}$ and hence $s_{\alpha}(\beta)\in\Psi$. Otherwise $\beta-s\alpha\in \Phi^{\text{re}}$ for all $0\leq s\leq (\beta,\alpha^{\vee})_X$. Since $-\alpha\in\Psi$ we get by the closedness of $\Psi$ that $\beta-s\alpha\in \Psi$. Thus $s_{\alpha}(\beta)\in\Psi$. The case $-(\beta,\alpha^{\vee})_X\in\bz_+$ works similarly and we omit the details. We will use this fact in the rest of this paper without further comment.
\end{rem}
Given a subroot system $\Psi$, we define the gradient to be the subset
$$\text{Gr}(\Psi)=\{\alpha \in  \mathring{\Phi}: \exists \ y\in\Lambda_{\a} \text{ such that } \alpha\oplus y\in \Psi\}.$$
Further, for $\alpha\in \text{Gr}(\Psi)$ we set $Z_{\alpha}(\Psi)=\{y\in \Lambda_{\alpha}: \alpha\oplus y\in \Psi\}$.
It is clear that $Z_{\alpha}(\Psi)$ is non--empty and
$$\Psi=\{\alpha\oplus y: \alpha\in \text{Gr}(\Psi), y\in Z_{\alpha}(\Psi)\}.$$
It follows that the subroot system $\Psi$ is uniquely determined by $\text{Gr}(\Psi)$ and $Z_{\alpha}(\Psi), \alpha\in \text{Gr}(\Psi)$. 
\begin{lem}\label{lemm1}Let $\Phi$ an affine reflection system and $\Psi$ a subroot system. Then  we have that $\text{Gr}(\Psi)\subseteq \mathring{\Phi}$ is a subroot system of the affine reflection system $\mathring{\Phi}\cup \{0\}$. Moreover,
 $$Z_{\beta}(\Psi)-(\beta,\alpha^{\vee})Z_{\alpha}(\Psi)\subseteq Z_{s_{\alpha}(\beta)}(\Psi),\ \ \forall \alpha,\beta\in \text{Gr}(\Psi).$$
\proof
Let $\alpha,\beta\in \text{Gr}(\Psi)$, then we have $(\alpha\oplus y_{\alpha}),(\beta\oplus y_{\beta})\in \Psi$ for some $y_\alpha\in Z_{\alpha}(\Psi)$ and $y_\beta\in Z_{\beta}(\Psi)$. In order to show that $s_{\alpha}(\beta)\in \text{Gr}(\Psi)$ we note that
$$s_{\alpha\oplus y_{\alpha}}(\beta \oplus y_{\beta})=s_{\alpha}(\beta)\oplus (y_{\beta}-(\beta,\alpha^{\vee})y_{\alpha})\in\Psi.$$
But the definition of an extension datum (see Definition~\ref{extda}) gives $(y_{\beta}-(\beta,\alpha^{\vee})y_{\alpha})\in \Lambda_{s_{\alpha}(\beta)}$ and thus $s_{\alpha}(\beta)\in \text{Gr}(\Psi)$, proving that the gradient is a subroot system. The second part of the lemma also follows.
\endproof
\end{lem}
\subsection{}\label{32}We introduce a $\bz$--linear function 
$$p: \text{Gr}(\Psi)\rightarrow \bigcup_{\a\in \mathring{\Phi}}\Lambda_{\a},\ \alpha\mapsto p_{\alpha}$$
as follows. By Lemma~\ref{lemm1} and \cite[Theorem 10.1]{Hu80} we can choose a simple system $\Pi\subseteq \text{Gr}(\Psi)$ and arbitrary $p_{\gamma}\in Z_{\gamma}(\Psi)$ for each $\gamma\in \Pi$. We extend this $\bz$--linearly and obtain
\begin{equation}\label{2}
p_{\beta}-(\beta,\alpha^{\vee})p_{\alpha}=p_{s_{\alpha}(\beta)}.
\end{equation} 
Now it is easy to see with \eqref{2} that $p_{\alpha}\in Z_{\alpha}(\Psi)$ for all $\alpha$ in the $W(\text{Gr}(\Psi))$--orbit of $\Pi$; equivalently $p_{\alpha}\in Z_{\alpha}(\Psi)$ for all $\alpha\in \text{Gr}(\Psi)\cap \mathring{\Phi}_{nd}$. We set $Z_{\alpha}=Z_{\alpha}(\Psi)-p_{\alpha}$ for $\alpha\in \text{Gr}(\Psi)$ and observe that $0\in Z_{\alpha}$ for all $\alpha\in \text{Gr}(\Psi)\cap \mathring{\Phi}_{nd}$. 
 \begin{prop}\label{prop2a}Let $\Psi$ a subroot system of an affine reflection system $\Phi$. Then the collection $\mathcal{Z}=(Z_0:=0, Z_{\alpha}, \alpha\in \text{Gr}(\Psi))$ is an extension datum and the affine reflection system constructed from $\mathcal{Z}$ and $\text{Gr}(\Psi)$ as in \eqref{consaff} is isomorphic to $\Psi\cup\{0\}$.
 \proof In view of \eqref{2} and $0\in Z_{\alpha}$ for all $\alpha\in \text{Gr}(\Psi)\cap \mathring{\Phi}_{nd}$ the first part of the claim follows from Lemma~\ref{lemm1}. The second part is straightforward to check.
\endproof
 \end{prop}
 Note that $\text{Gr}(\Psi)$ is not neccessarily irreducible, but we can decompose $\text{Gr}(\Psi)$ into its irreducible components
 \begin{equation}\label{irrdec}
\text{Gr}(\Psi)=\text{Gr}(\Psi)^1\cup \cdots \cup\text{Gr}(\Psi)^k.
 \end{equation}
 From Proposition~\ref{prop2a} we know that $\mathcal{Z}^i=(Z_{\alpha}, \alpha\in \text{Gr}(\Psi)^i\cup\{0\})$ is also an extension datum. As a corollary, we can apply the results of Section~\ref{section2} and obtain that there are at most 3 possible subsets $Z^i_s, Z^i_{\ell}$ and $Z^i_d$  for each connected component (we define as usual $Z_s^i:=Z_{\alpha}, \alpha\in \text{Gr}(\Psi)^i\cap \mathring{\Phi}_s$ etc.) satisfying the relations stated in Lemma~\ref{inclusionen}. The following lemma will be needed later.
 \begin{lem}\label{lem3}Let $\Phi$ an untwisted irreducible affine reflection system and $\Psi$ a closed subroot system (resp. maximal closed subroot system) of $\Phi$. Then we have that $\text{Gr}(\Psi)$ is a closed subroot system (resp. maximal closed subroot system) and
 $$Z_{\alpha}=Z_{\beta},\ \forall \alpha,\beta\in \text{Gr}(\Psi)^i,\ 1\leq i\leq k.$$
 \proof  We already know from Lemma~\ref{lemm1} that  $\text{Gr}(\Psi)$ is a subroot system. In order to show that $\text{Gr}(\Psi)$ is closed, let $\alpha\in \mathring{\Phi}_a\cap \text{Gr}(\Psi)	,\ \beta\in\mathring{\Phi}_b\cap \text{Gr}(\Psi)$ such that $\alpha+\beta\in \mathring{\Phi}_c$, where $a,b,c\in\{s,\ell,d\}$. We shall prove that $\alpha+\beta\in \text{Gr}(\Psi)$ by showing that $(\alpha+\beta)\oplus (y_{\alpha}+y_{\beta})\in \Phi^{\text{re}}$, where $y_\alpha\in Z_{\alpha}(\Psi)$ and $y_\beta\in Z_{\beta}(\Psi)$ are such that $(\alpha\oplus y_{\alpha}),(\beta\oplus y_{\beta})\in \Psi$. But this follows from Lemma~\ref{inclusionen}, since $\Lambda_s=\Lambda_\ell=\Lambda_d$ is a subgroup. Now assume that $\Psi$ is in addition maximal and $\text{Gr}(\Psi)\subseteq \mathring{\Phi}'\subsetneq \mathring{\Phi}$, where $\mathring{\Phi}'$ is a closed subroot system. Then we define
 $$\Psi'=\{\alpha\oplus y: \alpha\in \mathring{\Phi}',\ y\in \Lambda_{\alpha}\}.$$
 Obviously $\Psi\subseteq \Psi'\subsetneq \Phi$ and $\Psi'$ is a closed subroot system. By the maximality of $\Psi$ we obtain $\text{Gr}(\Psi)=\mathring{\Phi}'$ and hence $\text{Gr}(\Psi)$ is maximal. For the remaining part of the claim it is enough by Lemma~\ref{inclusionen} and the discussion preceeding the lemma to show that $Z^i_s\subseteq Z^i_{x}$, where $x=\ell$ in the reduced case and $x=d$ otherwise. This is clear if $\text{Gr}(\Psi)^i$ is of simply--laced type. Otherwise we can choose $\alpha,\beta$ to be short roots in $\text{Gr}(\Psi)^i$ such that $\alpha+\beta$ is a long root (resp. divisible root in the non--reduced case) by Lemma~\ref{hum}(2). Thus we get
$$Z^i_s+Z^i_s=Z^i_{\alpha}+Z^i_{\beta}\subseteq Z^i_{\alpha+\beta}=Z^i_{x},$$
because if $y_{\alpha}\in Z^i_{\alpha}(\Psi)$, $y_{\beta}\in Z^i_{\beta}(\Psi)$ (defined in the ovious way), then 
$$(\alpha\oplus y_{\alpha})+(\beta\oplus y_{\beta})=(\alpha+\beta)\oplus (y_{\alpha}+y_{\beta})\in\Phi^{\text{re}}$$
which implies $(\alpha+\beta)\oplus (y_{\alpha}+y_{\beta})\in\Psi$ and hence $y_{\alpha}+y_{\beta}\in Z^i_{\alpha+\beta}(\Psi)$ and the claim follows. 
\endproof
\end{lem}

\subsection{}\label{33}It can happen that the gradient of a maximal closed subroot system is not closed in general, but still we will have some control. The following definition is motivated by Lemma~\ref{lem3}. 
\begin{defn}\label{semicl}We call a subroot system $\mathring{\Phi}'$ of $\mathring{\Phi}$ semi--closed, if it is not closed and $\alpha,\beta\in \mathring{\Phi}'$ such that $\alpha+\beta\in \mathring{\Phi}\backslash \mathring{\Phi}'$ implies one of the following situations
\begin{itemize}
\item $\alpha,\beta \in \mathring{\Phi}_s,\ \ \alpha+\beta\in\mathring{\Phi}_\ell \cup \mathring{\Phi}_d.$
\item $\alpha,\beta \in \mathring{\Phi}_\ell,\ \ \alpha+\beta\in \mathring{\Phi}_d.$
\end{itemize}

Moreover, we call $\mathring{\Phi}'$ maximal semi--closed if it is semi--closed and $\mathring{\Phi}'$ is not properly contained in a proper closed subroot system of $\mathring{\Phi}$.
\end{defn}
In particular, if $\Phi$ is an irreducible affine reflection system and $\Psi$ a closed subroot system, then $\text{Gr}(\Psi)$ must be closed or semi--closed. To see this, assume that it is not closed, i.e. there exists $\alpha,\beta\in \text{Gr}(\Psi)$ such that $\alpha+\beta\in \mathring{\Phi}\backslash \text{Gr}(\Psi)$ (recall the possibilities from Lemma~\ref{hum}). Let $y_{\alpha}\in Z_\a(\Psi), y_{\beta}\in Z_\beta(\Psi)$ such that $(\alpha\oplus y_{\alpha})\in \Psi$ and $(\beta\oplus y_{\beta})\in \Psi$, but $(\alpha+\beta)\oplus (y_{\alpha}+y_{\beta})\notin \Phi^{\text{re}}$. By Lemma~\ref{inclusionen} this is only possible if both roots are short and the sum is a long (resp. divisible) root or both roots are long and the sum is a divisible root. Moreover, if $\Psi$ is a maximal closed subroot system such that $\text{Gr}(\Psi)$ is semi--closed the same proof as in Lemma~\ref{lem3} shows that $\text{Gr}(\Psi)$ is in fact maximal semi--closed.

\begin{example}\label{ex2}
If $\Psi$ is a closed subroot system of an affine root system as explained in Example~\ref{ex1}(2), we have that $\text{Gr}(\Psi)$ is closed if $m=1$. This follows immediately from Lemma~\ref{lem3}. If $m>1$ it can also be semi--closed. For example, in the case $(A_{2n-1},2)$ we have that $\mathring{\Phi}$ is of type $C_n$. Let $\alpha_n$ be the unique long root. Then
$$\Psi_1:=\{\alpha\oplus 2\bz \delta,\ \alpha\in\mathring{\Phi}\},\ \ \Psi_2:=\{\pm \alpha_{n-1}\oplus 2\bz \delta,\ \pm (\alpha_{n-1}+\alpha_n)\oplus (2\bz+1)\delta \}$$
are closed subroot system and $\text{Gr}(\Psi_1)=\mathring{\Phi}$ is closed whereas 
$\text{Gr}(\Psi_2)$ is only semi--closed. 
\end{example}
\subsection{}We have seen in the previous subsection the importance of maximal closed and maximal semi--closed subroot systems of (finite) root systems. Borel and de Siebenthal classified the maximal closed subroot systems \cite{BdS} in the case when $\mathring{\Phi}$ is reduced. The concern with closed subroot systems of $\mathring{\Phi}$ arises from the theory of finite--dimensional semi--simple Lie algebras, namely determining the closed subroot systems of a root system amounts to determining the semi--simple Lie subalgebras of a semi--simple Lie algebra. Their main result of Borel and de Siebenthal can be summarized as follows.
\begin{thm}\label{BdS}Let $\mathring{\Phi}$ be a reduced irreducible root system with fundamental system $\{\alpha_1,\dots,\alpha_n\}$ and highest root $$\alpha_0=\sum_{i=1}^n a_i\alpha_i.$$ The maximal closed subroot systems (up to $W(\mathring{\Phi})$--conjugacy) are those with fundamental system
\begin{enumerate}
\item $\{\alpha_1,\dots,\widehat{\alpha_i},\dots,\alpha_n\}$, where $a_i=1$ or
\item $\{-\alpha_0,\alpha_1,\dots,\widehat{\alpha_i},\dots,\alpha_n\}$, where $a_i$ is prime.
\end{enumerate}\hfill\qed
\end{thm}
The above theorem also implies how the maximal closed subroot systems of non--reduced irreducible root systems look like. If $\mathring{\Phi}$ is a non--reduced irreducible root system and $\mathring{\Phi}'$ a maximal closed subroot system, one can easily check the following two facts
\begin{itemize}
\item $(\mathring{\Phi}'\cup \mathring{\Phi}_d)$ is a proper closed subset of $\mathring{\Phi}$ and hence the maximality gives $\mathring{\Phi}_d\subseteq \mathring{\Phi}'$.
\item $\mathring{\Phi}'\cap (\mathring{\Phi}_s\cup \mathring{\Phi}_\ell)$ is a maximal closed subroot system in the reduced root system $(\mathring{\Phi}_s\cup \mathring{\Phi}_\ell)$.
\end{itemize} So we obtain the following.
\begin{cor}\label{BdSnr}
Let $\mathring{\Phi}$ be a non--reduced irreducible root system. The maximal closed subroot systems (up to $W(\mathring{\Phi})$--conjugacy) are those 
$$\mathring{\Phi}_d\cup \mathring{\Phi}',$$
where $\mathring{\Phi}'$ is a maximal closed subroot system of the reduced root system $(\mathring{\Phi}_s\cup \mathring{\Phi}_\ell)$.
\hfill\qed
\end{cor}
\subsection{} The rest of this section is dedicated to the study of the gradient, when the affine reflection system is twisted. We record the following important property of maximal semi--closed subroot systems.

\begin{lem}\label{shortarethere}Let $\mathring{\Phi}$ be a reduced irreducible root system and $\mathring{\Phi}'$ a maximal semi--closed subroot systems. Then 
\begin{equation}\label{conclg}\mathring{\Phi}'\cup \{\gamma\in \mathring{\Phi}_{\ell}: \gamma\in (\mathring{\Phi}'\cap \mathring{\Phi}_s)+(\mathring{\Phi}'\cap\mathring{\Phi}_s)\}
\end{equation}
is the closed subroot system generated by $\mathring{\Phi}'$, i.e. the closed subroot system which contains $\mathring{\Phi}'$ and is minimal with this property. In particular,
we obtain $\mathring{\Phi}_s\subseteq \mathring{\Phi}'$.
\proof Denote by $<\mathring{\Phi}'>$ the closed subroot system generated by $\mathring{\Phi}'$ and let $\mathring{\Phi}'(1)$ denote the subset \eqref{conclg}. It is clear that $<\mathring{\Phi}'>$ contains $\mathring{\Phi}'(1)$ and equality would follow by Remark~\ref{automatic} if we can show that $\mathring{\Phi}'(1)$ is a closed subset in $\mathring{\Phi}$. This will be the aim for the rest of the proof.

\vspace{0,2cm}
\textit{Case 1:} Suppose that $\gamma=(\tau_1+\tau_2)\in\mathring{\Phi}_{\ell},$ where $\tau_1,\tau_2\in  (\mathring{\Phi}'\cap \mathring{\Phi}_s)$ and $\tau_3\in \mathring{\Phi}'$ are such that $(\gamma+\tau_3)\in \mathring{\Phi}$. In what follows we will show that $(\gamma+\tau_3)\in \mathring{\Phi}'(1)$. Since $\gamma\in \mathring{\Phi}_{\ell}$ we get
$$2(\gamma,\tau_3)=(\gamma+\tau_3,\gamma+\tau_3)-(\gamma,\gamma)-(\tau_3,\tau_3)<0$$ and hence $(\tau_2,\tau_3)<0$ or $(\tau_1,\tau_3)<0$.

\vspace{0,2cm}
\textit{Case 1.1:} Assume that $(\tau_2,\tau_3)<0$. By \cite[Lemma 9.4]{Hu80} we have $\tau_2+\tau_3\in \mathring{\Phi}$ (if the roots are proportional there is nothing to show). If $(\tau_2+\tau_3)\in \mathring{\Phi}_s$, we get by the semi--closedness of $\mathring{\Phi}'$ that $(\tau_2+\tau_3)\in \mathring{\Phi}'$ and hence $(\gamma+\tau_3)=\tau_1+(\tau_2+\tau_3)\in \mathring{\Phi}'(1)$. So assume that $(\tau_2+\tau_3)\in\mathring{\Phi}_\ell$. In this case we only have to show that $\tau_3\in\Phi_{\ell}$, because the semi--closedeness would again imply $(\tau_2+\tau_3)\in \mathring{\Phi}'$ and hence $(\gamma+\tau_3)=\tau_1+(\tau_2+\tau_3)\in \mathring{\Phi}'(1)$. Assume by contradiction that $(\tau_2+\tau_3)\in\mathring{\Phi}_\ell$ and $\tau_3\in\mathring{\Phi}_{s}$. Since $\gamma$ is a long root we have
$$(\gamma,\gamma)-(\tau_1,\tau_1)-(\tau_2,\tau_2)=2(\tau_1,\tau_2)\geq 0.$$ But on the other hand since $(\tau_1,\tau_1)=(\tau_3,\tau_3)$ and $(\gamma,\gamma)=(\tau_2+\tau_3,\tau_2+\tau_3)$ we obtain  $(\tau_1,\tau_2)=(\tau_2,\tau_3)<0$, which is a contradiction.

\vspace{0,2cm}
\textit{Case 1.2:} Assume that $(\tau_1,\tau_3)<0$. The proof is exactly as in Case 1.1 by interchanging the role of $\tau_1$ and $\tau_2$. We omit the details.

\vspace{0,2cm}
\textit{Case 2:}  Let $\gamma_1=(\tau_1+\tau_2)\in\mathring{\Phi}_\ell$, $\gamma_2=\tau_3+\tau_4\in\mathring{\Phi}_\ell$, where $\tau_i\in  (\mathring{\Phi}'\cap \mathring{\Phi}_s)$ and $(\gamma_1+\gamma_2)\in \mathring{\Phi} $. We will show that $(\gamma_1+\gamma_2)\in \mathring{\Phi}'(1).$ For this it will be enough to show that $(\tau_1+\tau_2+\tau_3)\in\mathring{\Phi}$ or $(\tau_1+\tau_2+\tau_4)\in \mathring{\Phi}$, because then we can use Case 1 to finish the proof. Consider 
$$(\gamma_1+\gamma_2,\gamma_1+\gamma_2)-(\gamma_1,\gamma_1)-(\gamma_2,\gamma_2)=2(\gamma_1,\gamma_2)<0.$$
This implies $(\tau_1+\tau_2,\tau_3)<0$ or $(\tau_1+\tau_2,\tau_4)<0$ and we are done with \cite[Lemma 9.4]{Hu80}. 

Hence the closedness of $\mathring{\Phi}'(1)$ is established. Since $\mathring{\Phi}'$ is semi--closed, we have $\mathring{\Phi}'\subsetneq <\mathring{\Phi}'>$ and if $\mathring{\Phi}_s \not\subseteq \mathring{\Phi}'$ we would also have that $<\mathring{\Phi}'>\subsetneq \mathring{\Phi}$. This contradicts the maximality of $\mathring{\Phi}'$ and thus it must contain all short roots.
\endproof
\end{lem}
\begin{rem}\label{rem2}
In type $C_n$, $B_2$ or $G_2$ we actually have $\mathring{\Phi}'=\mathring{\Phi}_s$ in Lemma~\ref{shortarethere}, i.e. the set of short roots is the unique maximal semi--closed subroot system. To see this, assume that $\mathring{\Phi}'$ is a maximal closed subroot system in the root system of type $C_n$, which contains a long root  of the form $2\epsilon_i$ for some $i\in\{1,\dots,n\}$. Since the group generated by the reflections $s_{\alpha}\in W(\mathring{\Phi})$ corresponding to the simple short roots generate a subgroup isomorphic to the symmetric group $S_n$ we get that $\mathring{\Phi}'$ contains all long roots. Hence $\mathring{\Phi}'$ is closed and we obtain our desired contradiction. A similar argument shows the claim for $B_2$ and $G_2$.
\end{rem}

\section{Classification - Part 1: The semi--closed reduced case}\label{section4}
The aim of this section is to classify all maximal closed subroot systems $\Psi$ of $\Phi$ satisfying the property that $\text{Gr}(\Psi)$ is semi--closed and $\mathring{\Phi}$ is reduced. Note that $\mathring{\Phi}$ cannot be of simply--laced type and $\Lambda_\ell\subsetneq \Lambda_s$ by Lemma~\ref{lem3}.
\subsection{} \textit{We assume in this subsection that $\mathring{\Phi}$ is of type $C_n$ or $G_2$.} Recall from Lemma~\ref{inclusionen} that $\Lambda_s$ is a subgroup in this case and set $\Gamma_s=\{(\a,\beta)\in\mathring{\Phi}_s\times \mathring{\Phi}_s: \alpha+\beta\in\mathring{\Phi}_\ell\}$. It is clear that $\Gamma_s$ is non--empty.
\begin{thm}\label{Cng2case}Let $\Phi$ an irreducible affine reflection system with $\mathring{\Phi}$ of type $C_n$ or $G_2$. We have that $\Psi\subseteq \Phi$ is a maximal closed subroot system with semi--closed gradient if and only if there exists a maximal subgroup $H\subseteq\Lambda_s$ of index $m_{\mathring{\Phi}}$ which contains $\Lambda_\ell$ and a $\mathbb{Z}$--linear function $\tau: \mathring{\Phi}_s \to \Lambda_s/H$ satisfying $\tau(\a)\neq -\tau(\beta)$  for all $\alpha,\beta\in \Gamma_s$ such that 
$$\Psi=\{\a \oplus \tau (\a):  \a\in\mathring{\Phi}_s\}.$$
 \end{thm}
\begin{proof}
Assume first that $\Psi\subseteq  \Phi$ is a maximal closed subroot system. We know from Lemma~\ref{shortarethere} and Remark~\ref{rem2} that $\mathrm{Gr}(\Psi)=\mathring{\Phi}_s$. Therefore, the gradient is irreducible and the discussion below \eqref{irrdec} implies that there exists a subgroup $Z_s\subseteq \Lambda_s$ such that
 $$\Psi=\{\alpha \oplus  (p_\a+Z_s): \alpha\in \mathring{\Phi}_s\},$$
 where we recall that $p: \mathring{\Phi}_s \to \Lambda_s$ is the $\bz$--linear function defined in Section~\ref{32}.
 We claim that $(Z_s+m_{\mathring{\Phi}}\Lambda_s)$ must be a proper subgroup of $\Lambda_s$. If not, we have $Z_s+m_{\mathring{\Phi}}\Lambda_s=\Lambda_s$ and we can assume for a moment that $p_\a\in m_{\mathring{\Phi}}\Lambda_s$ for all $\a\in\mathring{\Phi}_s$. Since $0\in Z_s$ and $m_{\mathring{\Phi}}\Lambda_s\subseteq \Lambda_\ell$ we get a contradiction to
 $$((p_{\a}+Z_s)+(p_{\beta}+Z_s)\big)\cap \Lambda_\ell =\emptyset,\ \ (\alpha,\beta)\in \Gamma_s.$$
Thus $(Z_s+m_{\mathring{\Phi}}\Lambda_s)$ is a proper subgroup of $\Lambda_s$ and we can choose a maximal subgroup $H$ of $\Lambda_s$ which contains the group generated by $Z_s$ and $m_{\mathring{\Phi}}\Lambda_s$. In particular, $\Lambda_s/H$ is a simple $\bz/m_{\mathring{\Phi}}\bz$--module and therefore $H$ is a subgroup of index $m_{\mathring{\Phi}}$. We set 
 $$\Psi'=\{\alpha \oplus  \tau(\a): \alpha\in \mathring{\Phi}_s\},$$
 where $\tau(\a)$ denotes the coset of $p_\a$ in $\Lambda_s/H$. Note that $\Psi\subseteq \Psi'$. We claim that 
 \begin{equation}\label{22}
 (\tau(\a)+\tau(\beta))\cap \Lambda_\ell =\emptyset,\ \ (\alpha,\beta)\in \Gamma_s.
 \end{equation}
 Assume by contradiction that \eqref{22} is false, i.e. $\Psi'$ is not closed and the closed subroot system $\Psi''$ generated by $\Psi'$ has the property $\text{Gr}(\Psi'')\cap \mathring{\Phi}_\ell\neq \emptyset$. Since $H$ is a proper subgroup of $\Lambda_s$, we have
 $\Psi''\subsetneq \Phi$ and hence $\Psi=\Psi''$ which is a contradiction to
 $\mathrm{Gr}(\Psi)=\mathring{\Phi}_s.$ This implies \eqref{22} which gives that $\Psi'$ is a closed subroot system. It follows by the maximality of $\Psi$ that $\Psi=\Psi'$. It remains to show the properties of the function $\tau$ and that $\Lambda_\ell$ is contained in $H$; recall that $H$ is of index $m_{\mathring{\Phi}}$ and $H\cap \Lambda_\ell\neq \emptyset$. If $\tau(\a)=-\tau(\beta)$ for some $(\a,\beta)\in\Gamma_s$ we get 
 $$0\in (\tau(\a)+\tau(\beta))\cap \Lambda_\ell,$$
 which is a contradiction to \eqref{22}. Hence $\tau(\a)\neq -\tau(\beta)$ for all $(\a,\beta)\in\Gamma_s$. We claim that we can choose
short roots $\beta_1,\beta_2,\dots, \beta_{m_{\mathring{\Phi}}}$ such that $(\beta_i,\beta_{i+1})\in\Gamma_s$ for all $1\leq i\leq m_{\mathring{\Phi}}-1$ and
 \begin{equation}\label{zuzhg}
 \{\tau(\beta_i)+\tau(\beta_{i+1}),\ 1\leq i< m_{\mathring{\Phi}}\}=(\Lambda_s/H)\backslash\{0\}.
 \end{equation} 
To see this, note that $m_{\mathring{\Phi}}\in \{2,3\}$, so if $m_{\mathring{\Phi}}=2$ equation \eqref{zuzhg} is immediate. If $m_{\mathring{\Phi}}=3$ equation \eqref{zuzhg} follows by a straightforward analysis of the root system of type $G_2$ using  $\tau(\a)\neq -\tau(\beta)$ for all $(\a,\beta)\in\Gamma_s$; we omit the details.
Now we have $\Lambda_\ell\subseteq H$, because $\Lambda_\ell$ lies in the complement of each coset $\tau(\beta_i)+\tau(\beta_{i+1})$ by \eqref{22} and by \eqref{zuzhg} we exhaust all cosets different from $H$. So the proof of the forward direction is done. 

For the converse direction assume we are given $$\Psi=\{\alpha \oplus  \tau(\a): \alpha\in \mathring{\Phi}_s\}$$
with $\tau$ and $H$ as required. If $(\a,\beta)\in \Gamma_s$ we have $(\tau(\a)+\tau(\beta))\cap \Lambda_\ell=\emptyset$, because $\Lambda_\ell\subseteq H$ and $\tau(\a)+\tau(\beta)\neq 0$. It follows that $\Psi$ is a closed subroot system. Assume that $\Psi\subseteq \Psi'\subseteq \Phi$ and $\Psi'$ is a maximal closed subroot system. If $\text{Gr}(\Psi')$ is semi--closed, the calculation above shows that there exists a maximal subgroup $H'$ of index $m_{\mathring{\Phi}}$ and a $\bz$--linear  function $\tau':\mathring{\Phi}_s\rightarrow \Lambda_s/H'$ such that $\tau'(\alpha)\neq -\tau'(\beta)$ for all $(\alpha,\beta)\in\Gamma_s$ and
$$\Psi'=\{\alpha \oplus  \tau'(\a): \alpha\in \mathring{\Phi}_s\}.$$
Since $\Psi\subseteq \Psi'$ we obtain $\tau(\a)\subseteq \tau'(\a)$ for all $\alpha\in \mathring{\Phi}_s$. With other words, we can choose representatives $\tau_\alpha,\tau_{\a'}\in \Lambda_s$ such that 
$\tau(\alpha)=\tau_\a+H,\ \tau'(\a)=\tau'_{\a}+H' \text{ and } (\tau_\a-\tau'_{\a})+H\subseteq H'$. Since $0\in H$ and $H'$ is a subgroup we get $H\subseteq H'$ and the maximality of $H'$ implies $H=H'$. Using $\tau(\a)\subseteq \tau'(\a)$ for all $\alpha\in \mathring{\Phi}_s$ we also get $\tau=\tau'$ and hence $\Psi=\Psi'$. If $\text{Gr}(\Psi')$ is closed we must have $\text{Gr}(\Psi')=\mathring{\Phi}$ (since $\mathring{\Phi}_s\subseteq \text{Gr}(\Psi')$ and $\mathring{\Phi}_\ell\subseteq \mathring{\Phi}_s+\mathring{\Phi}_s$). Hence $\text{Gr}(\Psi')$ is irreducible and we can find by the discussion in Section~\ref{32} a subgroup $Z_s^{\Psi'}\subseteq \Lambda_s$ and a subset $Z_\ell^{\Psi'}\subseteq \Lambda_\ell$ and a $\bz$--linear function $\tilde \tau:\text{Gr}(\Psi')\rightarrow \Lambda_s$ , $\alpha\mapsto \tilde\tau_\a$ such that $\tilde \tau(\mathring{\Phi}_\ell)\subseteq \Lambda_\ell$ and 
$$\Psi'=\{\alpha \oplus  (\tilde \tau_\alpha+Z_s^{\Psi'}): \alpha\in \mathring{\Phi}_s\}\cup \{\alpha \oplus  (\tilde \tau_\alpha+Z_\ell^{\Psi'}): \alpha\in \mathring{\Phi}_\ell\}.$$
 As above we can deduce from $\Psi\subseteq \Psi'$ that $H\subseteq Z_s^{\Psi'}$. By the maximality of $H$ we must have $H=Z_s^{\Psi'}$ or $Z_s^{\Psi'}=\Lambda_s$. The first case implies as before $\tau(\a)=\tilde\tau_\alpha+H$ for all $\a\in \mathring{\Phi}_s$ and contradicts the following fact. Let $(\alpha,\beta)\in\Gamma_s$ and recall that $\gamma:=\alpha+\beta\in \text{Gr}(\Psi')$; let $y\in\Lambda_\ell$ such that $\gamma\oplus y\in \Psi'$. Then
 $$(\gamma\oplus y)+(-\alpha\oplus -\tau(\alpha))=\beta\oplus -\tau(\alpha)\subseteq\Psi'$$
 and we get a contradiction to $\tau(\beta)\neq -\tau(\a)$. Therefore $\Lambda_s=Z_s^{\Psi'}$. But this implies easily $\Psi'=\Phi$, because if $\gamma$ is a long root such that $\gamma=\alpha+\beta$, $(\alpha,\beta)\in\Gamma_s$ we get
 $$\gamma\oplus \Lambda_\ell=(\alpha\oplus\Lambda_\ell)+(\beta\oplus 0)\subseteq\Psi'$$ and the maximality of $\Psi$ is established.
\end{proof}
\subsection{} \textit{We assume in this subsection that $\mathring{\Phi}$ is of type $F_4$.} In this case we have by Lemma~\ref{inclusionen} that $\Lambda_s$ and $\Lambda_\ell$ are both subgroups. Moreover, since $\Lambda_{\ell}\subseteq \Lambda_s$ and $2\Lambda_s\subseteq\Lambda_{\ell}$ we can write $\Lambda_s$ as a union of cosets
$$\Lambda_s=\dot{\bigcup}_{r\in T} (x_r+\Lambda_\ell),$$
where $T$ is some index set containing zero.
\begin{thm}\label{F4case} Let $\Phi$ an irreducible affine reflection system with $\mathring{\Phi}$ of type $F_4$. We have that $\Psi\subseteq \Phi$ is a maximal closed subroot system with semi--closed gradient if and only if there exists a maximal subgroup $H\subseteq \Lambda_s$  of index 2 containing $\Lambda_\ell$ and a cardinality two subset $I\subseteq\{1,2,3,4\}$ and a $\mathbb{Z}$--linear function $\tau: \mathring{\Phi}_s \to \Lambda_s/H$ satisfying
$$\tau(\epsilon_i)=\tau(\epsilon_j),\  \text{ if } i,j\in I \text{ or } i,j\notin I,\ \  \tau(\epsilon_r)\neq \tau(\epsilon_s),\ \ r\in I, s\notin I$$
such that
\begin{equation}\label{formf4}\Psi=\big\{\a \oplus \tau(\a):  \a\in\mathring{\Phi}_s\}\cup \{\pm\epsilon_i\pm\epsilon_j\oplus \Lambda_\ell: i,j\in I \text{ or } i,j\notin I\}. \end{equation}

\proof
Let $\Psi$ be a maximal closed subroot system. First we claim that $\text{Gr}(\Psi)$ is irreducible. If $\text{Gr}(\Psi)=\mathring{\Phi}_s$ this is clear. Otherwise there must be a long root in the gradient and since $\text{Gr}(\Psi)$ is stable under the reflection with short roots we can assume without loss of generality that $\pm\epsilon_i\pm\epsilon_j\in \text{Gr}(\Psi)$ for some $i,j\in\{1,2,3,4\}$. Again by reflecting the above root at a short root of the form $(\pm \epsilon_1\pm\epsilon_2\pm\epsilon_3\pm\epsilon_4)/2$ we get $(\pm\epsilon_r\pm\epsilon_s)\in \text{Gr}(\Psi)$ where $\{i,j,r,s\}=\{1,2,3,4\}$. This implies that
\begin{equation}\label{formgg}
\text{Gr}(\Psi)=\mathring{\Phi}_s\cup \{\pm\epsilon_i\pm\epsilon_j, \pm\epsilon_r\pm\epsilon_s\},
\end{equation}
because the presence of any other long root in the gradient would imply $\text{Gr}(\Psi)=\mathring{\Phi}$ and hence contradicts the semi--closedness. Now we can directly show that $\text{Gr}(\Psi)$ is irreducible. Since the sum of two roots of two different connected components is never a root we must have that $\{\pm\epsilon_i,\pm\epsilon_j,\pm\epsilon_i\pm\epsilon_j\}$ must lie in the same connected component. Likewise $\{\pm\epsilon_r,\pm\epsilon_s,\pm\epsilon_r\pm\epsilon_s\}$ lies in the same connected component. Now take a short root of the form $(\pm \epsilon_1\pm\epsilon_2\pm\epsilon_3\pm\epsilon_4)/2$. This root must lie in the same connected component as $\epsilon_i$ as well as in the same connected component as $\epsilon_r$. Hence $\text{Gr}(\Psi)$ is irreducible.

 To simplify the notation we will assume in the rest of the proof that $\text{Gr}(\Psi)=\mathring{\Phi}_s$ or $\text{Gr}(\Psi)$ is of the form \eqref{formgg} with $I=\{1,2\}$. By the discussion in Section~\ref{32} there is a $\bz$--linear function $p:\text{Gr}(\Psi)\rightarrow \Lambda_s$ and subsets $Z_s\subseteq \Lambda_s$, $Z_\ell\subseteq \Lambda_\ell$ such that $Z_\ell\subseteq Z_s$ and $Z_{\alpha}(\Psi)=p_\a+Z_s$ if $\alpha\in \mathring{\Phi}_s$ and $Z_{\alpha}(\Psi)=p_\a+Z_\ell$ otherwise. We also note that $Z_s$ must be a group because the type of $\text{Gr}(\Psi)$ is either $D_4$ or $C_4$. We set
$$H=\bigcup_{r\in T'}(x_r+\Lambda_\ell),\ \ T'=\{r\in T: Z_s\cap(x_r+\Lambda_\ell)\neq \emptyset\}$$
and claim that
\begin{equation}\label{cl1}
\Psi=\{\alpha\oplus (p_\a+H): \a\in \mathring{\Phi}_s\}\cup \{\pm\epsilon_1\pm\epsilon_2\oplus \Lambda_\ell, \pm\epsilon_3\pm\epsilon_4\oplus \Lambda_\ell\}.
\end{equation}
Since $Z_s\subseteq H$ we obtain that $\Psi$ is contained in the right hand side of \eqref{cl1} and hence we only have to show that the right hand side is a closed subroot system. Since $Z_s$ is a group it is not difficult to show that $H$ is also a group. We give only a sketch for the closedness. Assume that $\a,\beta\in \mathring{\Phi}_s$ such that $(\a+\beta)$ is a long root but $(\a+\beta)\notin \{\pm\epsilon_1\pm\epsilon_2, \pm\epsilon_3\pm\epsilon_4\}$. So we have to show that for every choice of $z_1,z_2\in H$ we have $$p_{\alpha}+p_\beta+z_1+z_2\notin \Lambda_\ell.$$  But this is clear since there is always $z_1',z_2'\in Z_s$ such that $z_i\equiv z_i' \mod \Lambda_\ell$, $i=1,2$; so if $p_{\alpha}+p_\beta+z_1+z_2\in \Lambda_\ell$ we would also have $p_{\alpha}+p_\beta+z'_1+z'_2\in \Lambda_\ell$, which is a contradiction to the closedness of $\Psi$. The rest becomes clear and we get \eqref{cl1}. So $H=Z_s$ is a group such that $\Lambda_\ell \subseteq H\subseteq \Lambda_s$. We set $\tau(\a)=p_\a+H$. Since $\Psi$ is closed, it is straightforward to check that
$$\tau(\epsilon_1)=\tau(\epsilon_2),\ \ \tau(\epsilon_3)=\tau(\epsilon_4),\ \ \tau(\epsilon_1)\neq \tau(\epsilon_3).$$
Hence it remains to show that $H$ is a maximal group; the fact that $H$ is of index 2 would follow from the maximality and $2\Lambda_s\subseteq \Lambda_\ell\subseteq H$. If there is another group $\Lambda_\ell\subseteq H\subseteq H'\subseteq \Lambda_s$ we can have two cases. In the first case we assume $p_{\epsilon_1}+H'\neq p_{\epsilon_3}+H'$. In this case
$$\Psi\subseteq \{\alpha\oplus (p_\a+H'): \a\in \mathring{\Phi}_s\}\cup \{\pm\epsilon_1\pm\epsilon_2\oplus \Lambda_\ell, \pm\epsilon_3\pm\epsilon_4\oplus \Lambda_\ell\}$$
and the right hand side is closed. This gives $H'=H$. In the second case we assume $p_{\epsilon_1}+H'= p_{\epsilon_3}+H'$. In this case 
$$\Psi\subseteq \{\alpha\oplus (p_\a+H'): \a\in \mathring{\Phi}_s\}\cup \{\beta\oplus \Lambda_\ell: \beta\in \mathring{\Phi}_\ell\}$$ and the right hand side is closed. We get $H'=\Lambda_s$ and we are done. 

For the converse direction we assume that $\Psi$ is of the form \eqref{formf4} and $\Psi'$ is a maximal closed subroot system containing $\Psi$. If $\text{Gr}(\Psi')$ is semi--closed we know from the calculations above that there exists a group $H'$ and a $\bz$--linear function $\tau'$ with the required properties such that $\Psi'$ is of the form \eqref{formf4}. Similarly as in the proof of Theorem~\ref{Cng2case} the maximality of $H$ immediately gives $\Psi=\Psi'$. So assume that $\text{Gr}(\Psi')=\mathring{\Phi}$. Since $\Psi\subseteq \Psi'$ and every long root can be written as a sum of two long roots (where one of them is in $\text{Gr}(\Psi)$) we obtain $\{\beta\oplus \Lambda_\ell: \beta\in \mathring{\Phi}_\ell\}\subseteq \Psi'$. By the irreducibility we can again find a $\bz$--linear function $p'$ and a subgroup $Z_s^{\Psi'}$ such that $Z_{\alpha}(\Psi')=p'_\a+Z_s^{\Psi'}$. Since
$$\epsilon_1\oplus (p'_{\epsilon_1}+Z_s^{\Psi'})+(-\epsilon_1+\epsilon_2)\oplus 0=\epsilon_2\oplus (p'_{\epsilon_1}+Z_s^{\Psi'})\subseteq\Psi'$$
we must have $p'_{\epsilon_1}\equiv p'_{\epsilon_2}$ mod $Z_s^{\Psi'}$.
Continuing in this way we obtain that $p^{'}$ is constant modulo $Z_s^{\Psi'}$. Therefore, by adding two short roots such that the sum is again a short root we get $2p_\a'\in p_\a'+Z_s^{\Psi'}$ and thus $Z_{\alpha}(\Psi')=Z_s^{\Psi'}$. 
This implies $H\subseteq Z_s^{\Psi'}$ and since $p'_{\epsilon_1}-p'_{\epsilon_3}\in Z_s^{\Psi'}\backslash H$ we get
$$\Lambda_\ell\subseteq H\subsetneq Z_s^{\Psi'}\subseteq \Lambda_s.$$ But this is impossible by the maximality of $H$.
\endproof
\end{thm}

\subsection{}\textit{We assume in this subsection that $\mathring{\Phi}$ is of type $B_n$.} Recall that $\Lambda_{\ell}$ is a subgroup in this case (see also our mild assumption) and hence $(\Lambda_s\backslash\Lambda_{\ell})+\Lambda_\ell\subseteq (\Lambda_s\backslash\Lambda_{\ell})$. Again we can write $\Lambda_s$ as a union of cosets
as in the previous subsection.
\begin{thm}\label{Bncase} Let $\Phi$ an irreducible affine reflection system with $\mathring{\Phi}$ of type $B_n$. We have that $\Psi\subseteq \Phi$ is a maximal closed subroot system with semi--closed gradient if and only if there exists a proper non--empty subset $J\subseteq \{1,\dots,n\}$ and disjoint subsets $Z_1,Z_2\subseteq \Lambda_s$ such that $Z_1$ and $Z_2$ are  unions of cosets modulo $\Lambda_\ell$, $\Lambda_s=Z_1\cup Z_2$ and
\begin{equation}\label{Bnform}\Psi=\{\pm\epsilon_i\oplus Z_1, \pm\epsilon_j\oplus Z_2: i\notin J, j\in J\}\cup\{(\pm\epsilon_p\pm\epsilon_q)\oplus \Lambda_{\ell}: p,q\notin J \text{ or } p,q\in J\}.\end{equation}
\end{thm}
 \proof Let $\Psi$ a maximal closed subroot system and set 
$$I=\{i\in \{1,\dots,n\}: Z_{\epsilon_i}(\Psi)\cap \Lambda_{\ell}\neq \emptyset\},\ \ J=\{i\in \{1,\dots,n\}: Z_{\epsilon_i}(\Psi)\cap \Lambda_{\ell}= \emptyset\}.$$
We associate the sets $Z_{I}$ and $Z_{J}$ respectively as follows:
$$Z_{I}=\dot{\bigcup}_{r\in T(I)} (x_r+\Lambda_\ell),\ \ Z_{J}=\Lambda_s\backslash Z_{I},$$
where $T(I)=\{r\in T: \exists s\in I \text{ such that } Z_{\epsilon_s}(\Psi)\cap (x_r+\Lambda_\ell)\neq \emptyset\}.$
Since $\text{Gr}(\Psi)$ is not closed and stable under the reflection with short roots by Lemma~\ref{conclg} we can assume that there are short roots $\epsilon_i,\epsilon_j\in \text{Gr}(\Psi)$ such that $(\epsilon_i+\epsilon_j)\notin \text{Gr}(\Psi)$. Clearly $i\in J$ or $j\in J$; otherwise the closedness of $\Psi$ would give $(\epsilon_i+\epsilon_j)\in \text{Gr}(\Psi)$. We claim that 
\begin{equation}\label{1}
\Psi=\{\pm\epsilon_i\oplus Z_{I}, \pm\epsilon_j\oplus Z_{J}: i\in I, j\in J\}\cup\{(\pm\epsilon_p\pm\epsilon_q)\oplus \Lambda_{\ell}: p,q\in I \text{ or } p,q\in J\}.\end{equation}
 We first show that $\Psi$ is contained in the right hand side of \eqref{1}. Assume by contradiction that $(\epsilon_p+\epsilon_q)\in \text{Gr}(\Psi)$ for some $p\in I$, $q\in J$; let $y\in \Lambda_{\ell}$ such that $(\epsilon_p+\epsilon_q)\oplus y\in \Psi$. Since $p\in I$ we can choose $y_{p}\in\Lambda_{\ell}$ such that $(\epsilon_p\oplus y_p)\in \Psi$ and hence $$(\epsilon_p+\epsilon_q)-\epsilon_p\oplus (y-y_p)\in \Psi,$$
which is a contradiction to the choice of $q$. Clearly $Z_{\epsilon_i}(\Psi)\subseteq Z_{I}$ for all $i\in I$. Let $p\in J$ and $y\in Z_{\epsilon_p}(\Psi)$ such that $y\in Z_{I}$. This would imply that there exists $k\in T(I)$, $y',y''\in\Lambda_\ell$ and $i\in I$ such that $y=(x_k+y')$ and $(x_k+y'')\in Z_{\epsilon_i}(\Psi)$. Since $i\in I$ we can also choose $z\in Z_{\epsilon_i}(\Psi)\cap \Lambda_\ell$. Hence
$$(\epsilon_i+\epsilon_p)-\epsilon_i\oplus (2x_k+y'+y''-z)\in \Psi,$$
which contradicts again $p\in J$. Therefore $Z_{\epsilon_j}(\Psi)\subseteq Z_{J}$ for all $j\in J$. So we proved that $\Psi$ is contained in the right hand side of \eqref{1}.
In order to show equality we will prove that the right hand side is a proper closed subset of $\Phi$. Since $J$ is non--empty and $Z_{J}\subseteq \Lambda_s\backslash\Lambda_{\ell}$ it is definitely proper. The closedness follows from $(Z_{I}+Z_{J})\cap \Lambda_\ell=\emptyset$ and $Z_{I}+\Lambda_{\ell}\subseteq Z_{I}$,\ $Z_{J}+\Lambda_{\ell}\subseteq Z_{J}$. Thus \eqref{1} follows by the maximality of $\Psi$. 

For the converse direction let $\Psi$ as in \eqref{Bnform} and assume without loss of generality that $\Lambda_\ell\subseteq Z_1$. The same calculation as above shows that $\Psi$ is a closed subroot system. Suppose that $\Psi\subseteq\Psi'\subseteq \Phi$, where $\Psi'$ is a maximal subroot system. If $\text{Gr}(\Psi')$ is semi--closed we must have $\Psi=\Psi'$, since $\Psi'$ is by the above calculations again of the form \eqref{Bnform} and contains $\Psi$. Otherwise $\text{Gr}(\Psi')$ is closed and since $\mathring{\Phi}_\ell\subseteq \mathring{\Phi}_s+\mathring{\Phi}_s$ we have $\text{Gr}(\Psi')=\mathring{\Phi}$. Now noting that any $\gamma\in\mathring{\Phi}_\ell$ with $\gamma\notin \text{Gr}(\Psi)$ is of the form $\gamma=\pm\epsilon_p\pm \epsilon_q$ where $p\notin J, q\in J$ we immediately get $\gamma \oplus \Lambda_\ell\subseteq \Psi'$ since $\pm \epsilon_p\oplus \Lambda_\ell \subseteq \Psi'$. This already implies $\Psi'=\Phi$, since if $i\notin J$ we get
$$\epsilon_i\oplus Z_1 + (\epsilon_j-\epsilon_i)\oplus \Lambda_\ell\subseteq \Psi' \Rightarrow \epsilon_j\oplus Z_1\subseteq \Psi',\ \ \forall j\in J.$$  Similarly we can show $\epsilon_i\oplus Z_2\subseteq \Psi'$ for all $i\notin J$. 
\endproof

\section{Classification - Part 2: The full closed reduced case}\label{section5}

The aim of this section is to classify all maximal closed subroot systems $\Psi$ of $\Phi$ satisfying $\text{Gr}(\Psi)=\mathring{\Phi}$ and $\mathring{\Phi}$ is reduced.
\subsection{} \textit{We assume in this subsection that $\mathring{\Phi}$ is simply--laced or of type $C_n$, $F_4$ or $G_2$.} In this case we have by Lemma~\ref{inclusionen} (or our mild assumption) that $\Lambda_s$ is a subgroup (recall the convention $\Lambda_\ell=\Lambda_s$ in the simply--laced case). 
\begin{thm}\label{part2g2f4}Let $\Phi$ an irreducible affine reflection system with $\mathring{\Phi}$ of simply--laced type or of type $C_n$, $F_4$ or $G_2$. We have that $\Psi\subseteq \Phi$ is a maximal closed subroot system with $\text{Gr}(\Psi)=\mathring{\Phi}$ if and only if there exists a maximal subgroup $H\subseteq\Lambda_s$ and a $\mathbb{Z}$--linear function $\tau: \mathring{\Phi} \to \Lambda_s/H$ such that $\tau(\beta)\cap \Lambda_\ell\neq \emptyset$ for all $\beta\in \mathring{\Phi}_\ell$ and
\begin{equation}\label{4432}
\Psi=\{\a \oplus \tau(\a):  \a\in\mathring{\Phi}_s\}\cup \{\a \oplus (\tau(\a)\cap \Lambda_\ell):  \a\in\mathring{\Phi}_\ell\}.
\end{equation}
\proof
Let $\Psi\subseteq \Phi$ a maximal closed subroot system. Since $\text{Gr}(\Psi)=\mathring{\Phi}$ (and hence irreducible) we can find a subgroup $Z_s\subseteq \Lambda_s$ and a subset $Z_\ell\subseteq \Lambda_\ell$ such that $Z_\ell\subseteq Z_s$ and a $\bz$--linear function $p:\mathring{\Phi}\rightarrow \Lambda_s$, $\a\mapsto p_\a\in \Lambda_\a$ such that 
$$\Psi=\{\a \oplus p_\a+Z_s:  \a\in\mathring{\Phi}_s\}\cup \{\a \oplus p_\a+Z_\ell:  \a\in\mathring{\Phi}_\ell\}.$$
Clearly, $Z_s\subsetneq \Lambda_s$. Let $H\subseteq \Lambda_s$ a maximal subgroup containing $Z_s$ and set 
$$\Psi':=\{\a \oplus p_\a+H:  \a\in\mathring{\Phi}_s\}\cup \{\a \oplus (p_\a+H)\cap \Lambda_\ell:  \a\in\mathring{\Phi}_\ell\}.$$
Clearly $\Psi\subseteq \Psi'\subseteq \Phi$. Moreover, it is routine to check that $\Psi'$ is closed using Lemma~\ref{hum} and hence $\Psi=\Psi'$. Setting $\tau(\a)=p_\a+H$ for all $\a\in\mathring{\Phi}$ we get the desired property.
Conversely, suppose that $\Psi$ is of the form \eqref{4432} for some maximal subgroup $H$ and $\bz$--linear function $\tau$ such that $\Psi\subseteq \Psi' \subseteq \Phi$, where $\Psi'$ is a maximal closed subroot system. Since $\text{Gr}(\Psi')=\mathring{\Phi}$ we again have that $\Psi'$ is of the form \eqref{4432} for some maximal subgroup $H'$ and $\bz$--linear function $\tau'$. Hence $H\subseteq H'$
and the maximality of $H$ implies $H'=H$ (and therefore $\Psi=\Psi'$) or $H'=\Lambda_s$ (and therefore $\Psi'=\Phi$).
\endproof
\end{thm}
 \subsection{}\textit{We assume in this subsection that $\mathring{\Phi}$ is of type $B_n$.} In this case we have by Lemma~\ref{inclusionen} (or our mild assumption) that $\Lambda_\ell$ is a subgroup.
 \begin{thm}
\label{Bnformcasefull}Let $\Phi$ an irreducible affine reflection system with $\mathring{\Phi}$ of type $B_n$. We have that $\Psi\subseteq \Phi$ is a maximal closed subroot system with $\text{Gr}(\Psi)=\mathring{\Phi}$ if and only if 
\begin{equation}\label{prfcl00} \Psi=\{\a \oplus \tau(\a)+Z:  \a\in\mathring{\Phi}_s\}\cup \{\a \oplus \tau(\a)+H:  \a\in\mathring{\Phi}_\ell\},\end{equation}
where $H=\Lambda_\ell$ or $H\subseteq\Lambda_\ell$ is a maximal subgroup and $\tau: \mathring{\Phi} \to \Lambda_s$ is a $\mathbb{Z}$--linear function satisfying $\tau(\beta)\in \Lambda_\ell$ for all $\beta\in \mathring{\Phi}_\ell$ and $Z$ is a union of cosets modulo $H$ maximal with the property that $H\subseteq Z$, $Z=-Z$ and 
\begin{equation}\label{prfcl} (x+H), (y+H)\subseteq Z \text{ and } (x+H)\neq (y+H)\Rightarrow (x+y) \notin \Lambda_\ell.\end{equation}
 \proof
 Let $\Psi\subseteq \Phi$ be a maximal closed subroot system. Since $\text{Gr}(\Psi)=\mathring{\Phi}$ (and hence irreducible) we can find a subset $Z_s\subseteq \Lambda_s$ and a subgroup $Z_\ell\subseteq \Lambda_\ell$ such that $Z_\ell\subseteq Z_s$ and a $\bz$--linear function $\tau:\mathring{\Phi}\rightarrow \Lambda_s$ such that 
$$\Psi=\{\a \oplus \tau(\a)+Z_s:  \a\in\mathring{\Phi}_s\}\cup \{\a \oplus \tau(\a)+Z_\ell:  \a\in\mathring{\Phi}_\ell\}.$$
We set $H=\Lambda_\ell$ if $Z_\ell=\Lambda_\ell$ and otherwise we let $H$ a maximal subgroup of $\Lambda_\ell$ containing $Z_\ell$. Obviously we can write 
$$\Lambda_s=\dot{\bigcup}_{r\in S}(x_r+H),$$
for some index set $S$. We set $S'=\{r \in S: (x_r+H)\cap Z_s \neq \emptyset\}$ and define $Z$ as the union of cosets modulo $H$ with index set $S'$. First observe that if $x+H\subseteq Z$, then there exists $h\in H$ such that $x+h\in Z_s$. It follows $2x\in H$ and thus $x+H=-x+H$, which gives $Z=-Z$. Let us show that $Z$ satisfies \eqref{prfcl}. Let $\alpha,\beta$ short roots such that the sum is a long root. Further let $x+H, y+H\subseteq Z$ , i.e. there exists $h',h''\in H$ such that $x+h', y+h''\in Z_s$. Now suppose for contradiction that $x+y\in \Lambda_\ell$.  It follows $\tau(\a+\beta)+x+y+h'+h''\in \Lambda_\ell$ and hence $x+y+h'+h''\in Z_\ell$ (because $\Psi$ is closed). But this gives $x+y\in H$ and hence $x+H=-y+H=y+H$. It follows that $Z$ satisfies \eqref{prfcl}. Obviously,
$$\Psi\subseteq \Psi':=\{\a \oplus \tau(\a)+Z:  \a\in\mathring{\Phi}_s\}\cup \{\a \oplus  \tau(\a)+H:  \a\in\mathring{\Phi}_\ell\}\subseteq \Phi.$$
Moreover, $\Psi'$ is closed since $Z$ satisfies \eqref{prfcl}. Note that $\Psi'=\Phi$ would imply $H=\Lambda_\ell$ and thus $Z_\ell=\Lambda_\ell$. Now it is easy to see that we must have $Z_s=\Lambda_s$ in this case, which is a contradiction. So $\Psi'\subsetneq \Phi$ and the maximality of $\Psi$ implies $\Psi=\Psi'$.

 Conversely assume that $\Psi$ is of the form \eqref{prfcl00} and satisfies the desired properties and let $Z=\cup_{r\in T} (x_r+H)$ for some index set $T.$  If $H=\Lambda_\ell$, then condition \eqref{prfcl} is redundant and hence $\Psi$ is clearly a maximal closed subroot system by the maximality of $Z$. Let $H$ a proper maximal subgroup of $\Lambda_\ell$. Since $x+y\notin \Lambda_\ell$ implies $x+\Lambda_\ell \neq y+\Lambda_\ell$, we have $\{x_r+\Lambda_\ell : r\in T\}$ are distinct elements of $\Lambda_s/\Lambda_\ell.$
Write $\Lambda_s=\cup_{s\in S}(x_s+\Lambda_\ell)$, then we have $T\subseteq S.$ Now we see that $x_s+H\neq x_{s'}+H$ for all $s\neq s'\in S.$ Moreover we have
$x_s+x_{s'}\notin \Lambda_s$ for all $s\neq s'\in S.$ Hence $Z'=\cup_{s\in S}(x_s+H)$ satisfies \eqref{prfcl} and $Z\subseteq Z'\subsetneq \Lambda_s.$
Since $Z$ is maximal, we must have
$Z=Z'$. This implies $T=S$ and $\{x_r+\Lambda_\ell : r\in T\}$ exhaust $\Lambda_s.$ This in particular implies $Z+\Lambda_\ell=\Lambda_s.$ Now we are able to prove the maximality of $\Psi$.

Suppose $\Psi\subseteq \Psi'$ for some maximal closed subroot system $\Psi'.$ Then we know that there exists $\tau', Z', H'$ such that 
$$\Psi'=\{\a \oplus \tau'(\a)+Z':  \a\in\mathring{\Phi}_s\}\cup \{\a \oplus \tau'(\a)+H':  \a\in\mathring{\Phi}_\ell\}.$$
We must have $H=H'$ or $H'=\Lambda_\ell$ since $H\subseteq H'\subseteq \Lambda_\ell$ and $H$ is maximal. Suppose that $H=H'$. Since $Z\subseteq (\tau'(\a)-\tau(\a))+Z'$, $\a\in \mathring{\Phi}_s$, and $(\tau'(\a)-\tau(\a))+Z'$ satisfies also \eqref{prfcl} we get by the maximality of $Z$ that $\tau(\a)+Z=\tau'(\a)+Z'$ and thus $\Psi=\Psi'$. 
So assume that $H'=\Lambda_\ell.$ It follows that $Z'$ is a union of costes modulo $\Lambda_\ell$ and since $Z+\Lambda_\ell=\Lambda_s$ we obtain that $\tau'(\a)+Z'=\Lambda_s$ for all short roots $\a$. This implies $\Psi'=\Phi.$
 \endproof
 \end{thm}

\section{Classification - Part 3: The proper closed reduced case}\label{section6}
The aim of this section is to classify all maximal closed subroot systems $\Psi$ of $\Phi$ satisfying the property that $\text{Gr}(\Psi)$ is proper and closed and $\mathring{\Phi}$ is reduced.  Recall the classification of maximal closed subroot systems of finite root systems from Theorem~\ref{BdS} and Corollary~\ref{BdSnr}.
\subsection{} When the gradient is proper and closed the classification is simple.
\begin{thm}\label{propgrm}
Let $\Phi$ an irreducible affine reflection system with reduced finite root system $\mathring{\Phi}$. We have that $\Psi\subseteq \Phi$ is a maximal closed subroot system with proper closed gradient if and only if
\begin{equation}\label{propclo}
\Psi=\{\a \oplus \Lambda_s:  \a\in \mathring{\Phi}' \cap \mathring{\Phi}_s \}\cup \{\a \oplus \Lambda_\ell:  \a\in \mathring{\Phi}' \cap \mathring{\Phi}_\ell\},
\end{equation}
where 
\begin{enumerate}
\item $\mathring{\Phi}'$ is any maximal closed subroot system of $\mathring{\Phi}$ if $\Lambda_\ell=\Lambda_s$, 
\item $\mathring{\Phi}'$ is any maximal closed subroot system of $\mathring{\Phi}$ such that $\mathring{\Phi}' \cap \mathring{\Phi}_s \neq \emptyset$ if $\Lambda_\ell\subsetneq \Lambda_s$.
\end{enumerate}
\proof
Let $\Psi\subseteq \Phi$ a maximal closed subroot system with proper closed gradient. If $\text{Gr}(\Psi)$ is not a maximal closed subroot system of $\mathring{\Phi}$ we can find $\mathring{\Phi}'$ a maximal closed subroot system in $\mathring{\Phi}$ such that 
$$\Psi\subsetneq \{\a \oplus \Lambda_s:  \a\in \mathring{\Phi}'\cap \mathring{\Phi}_s\}\cup \{\a \oplus \Lambda_\ell:  \a\in \mathring{\Phi}'\cap \mathring{\Phi}_\ell\}.$$
The maximality of $\Psi$ implies $\mathring{\Phi}'=\mathring{\Phi}$, which is a contradiction. Hence $\text{Gr}(\Psi)$ is a maximal closed subroot system and clearly $\Psi$ must be of the form \eqref{propclo}. Now suppose in addition that $\Lambda_\ell\subsetneq \Lambda_s$
and $\text{Gr}(\Psi)\cap \mathring{\Phi}_s=\emptyset$. Then we get 
$$\Psi\subseteq \Psi':=\{\alpha\oplus \Lambda_\ell: \alpha\in \mathring{\Phi}_\ell\}\subsetneq \Phi$$
and $\Psi'$ is a closed subroot system in $\Phi$ (cf. Lemma~\ref{hum}). This implies $\text{Gr}(\Psi)=\mathring{\Phi}_\ell$ and hence $\mathring{\Phi}$ cannot be of type $C_n, n\geq 3$ (see Theorem~\ref{BdS}). It follows from Lemma~\ref{inclusionen} that $\Lambda_\ell$ is a group and hence
$$\Psi'':=\{\alpha\oplus \Lambda_\ell: \alpha\in \mathring{\Phi}\}\subsetneq \Phi$$
is also a closed subroot system of $\Phi$. We get 
$$\Psi=\Psi' \subsetneq \Psi'' \subsetneq \Phi,$$
which is a contradiction and thus $\text{Gr}(\Psi)\cap \mathring{\Phi}_s\neq\emptyset$.

For the converse direction assume that $\Psi$ is of the form \eqref{propclo} and $\Psi\subseteq \Psi'$, where $\Psi'$ is a maximal closed subroot system. In the untwisted case, i.e. $\Lambda_\ell=\Lambda_s$, we are done since $\text{Gr}(\Psi')$ must be a closed subroot system by Lemma~\ref{lem3} and $\text{Gr}(\Psi)\subseteq\text{Gr}(\Psi')$. If $\Lambda_\ell\subsetneq \Lambda_s$ we will use the additional information  $\text{Gr}(\Psi)\cap \mathring{\Phi}_s\neq\emptyset$ to show that $\text{Gr}(\Psi')$ must be a closed subroot system. This would finish the proof as in the untwisted case. Assume for contradiction that $\text{Gr}(\Psi')$ is semi--closed. From the classification in Section~\ref{section4} we get that $Z_\alpha(\Psi')\subsetneq \Lambda_s$ for all short roots $\alpha\in \text{Gr}(\Psi')$, which is a contradiction to $\Psi\subseteq \Psi'$.
\endproof
\end{thm}

\section{Classification - Part 4: The non--reduced case}\label{section7}
The aim of this section is to classify all maximal closed subroot systems $\Psi$ of $\Phi$ where $\mathring{\Phi}$ is non--reduced of type $BC_n, n\ge 2$. 
\subsection{}
For $J\subseteq\{1,\dots, n\}$ define 
$$A_J:=\big\{\pm2\epsilon_i,\ \pm\epsilon_j:  i\in\{1,\dots n\},\ j\in J\big\}\cup 
\big\{\pm\epsilon_r\pm\epsilon_s: \ r, s\in J \mbox{ or } r,s\notin J,\ r\neq s \big\}$$ and denote by
$\wh{A_J}$ the lift of $A_J$ in $\Phi$, i.e. 
 $$\wh{A_J}=\{\alpha\oplus \Lambda_\a: \a\in A_J\}.$$
Note that $A_J$ is a closed subroot system of $BC_n$ of 
type $\text{$C_{n-r}\oplus BC_{r}$}$ if $|J|=r$. Hence, the lift $\wh{A_J}$ of $A_J$ is a closed subroot system in $\Phi$. Further, set 
$$\Phi^{B_n}=\{\a\oplus y : \a\in \mathring{\Phi}_s\cup \mathring{\Phi}_\ell,\ y\in \Lambda_\a \}\subseteq \Phi.$$
Our main result is the following. 
\begin{thm}\label{mainBCn}
Let $\Phi$ be an irreducible affine reflection system with $\mathring{\Phi}$ of type $BC_n$, $n\geq 2$. We have that $\Psi\subseteq \Phi$ is a maximal closed subroot system if and only if one of the following cases hold:
 \begin{enumerate}
 \item  $\Psi=\wh{A_{J}}, \text{ for some } \emptyset\subseteq J\subsetneq\{1,\dots, n\}$ and $J\neq\emptyset$ if $\Lambda_s\neq \Lambda_\ell$.
  \item There exists a maximal closed subroot system $\Psi'$ in $\Phi^{B_n}$ with semi--closed gradient such that
 $$\Psi=\Psi'\cup \{\pm 2\epsilon_i\oplus \Lambda_d: 1\leq i\leq n\}.$$
Recall that $\Psi'$ is by Theorem~\ref{Bncase} of the form \eqref{Bnform}.
 \item There exists a maximal closed subroot system $\Psi'$ in $\Phi^{B_n}$ with full gradient such that
 $$\Psi=\Psi'\cup \{\pm 2\epsilon_i\oplus (Z_{\pm\epsilon_i-\epsilon_{i+1}}(\Psi')+Z_{\pm\epsilon_i+\epsilon_{i+1}}(\Psi'))\cap\Lambda_d: 1\leq i\leq n\},$$
 where we understand $\epsilon_{n+1}=\epsilon_1$. Recall that $\Psi'$ is by Theorem~\ref{Bnformcasefull} of the form \eqref{prfcl00}.
 \end{enumerate}
\end{thm}
The proof of the above theorem will be given in the rest of this section.
\subsection{} First, we will analyze when $\wh{A_J}$ is a maximal closed subroot system.
\begin{prop}\label{bcnkeyprop01}
The lift $\wh{A_J}$ of $A_J$ in $\Phi$ is a maximal closed subroot system for all $\emptyset \neq J\subsetneq \{1,\dots, n\}$.
\end{prop}
\begin{proof}
Let $\Delta$ be a closed subroot system of $\Phi$ such that 
$\wh{A_J}\subsetneq\Delta\subseteq\Phi$. We have the following possibilities for the elements in $\Delta\backslash\wh{A_J}$:
$(i)\  \pm\epsilon_i\oplus y\in \Delta$ for some $i\notin J, y\in \Lambda_s,\
(ii)\ \pm\epsilon_i\pm \epsilon_j\oplus y\in \Delta$ for some $i\notin J, j\in J, y\in \Lambda_\ell.$

\textit{Case 1:} Suppose $\pm\epsilon_i\oplus y\in \Delta$ for some $i\notin J, y\in \Lambda_s.$
Then we have $(\pm\epsilon_i\oplus y)+(\pm\epsilon_j\oplus -y)=\pm\epsilon_i\pm\epsilon_j\oplus 0\in\Delta$ for all $j\in J$.
Now,  $$\pm\epsilon_i\oplus \Lambda_s=(\mp \epsilon_j\oplus \Lambda_s)+(\pm\epsilon_i\pm\epsilon_j\oplus 0)\subseteq \Delta.$$
For $s\notin J$ with $i\neq s$, we have
$$\text{$(\pm\epsilon_i\oplus \Lambda_s)+(\mp\epsilon_i\pm\epsilon_s\oplus 0)=\pm\epsilon_s\oplus \Lambda_s \subseteq \Delta$.}$$
This proves that $\pm \epsilon_s\oplus\Lambda_s \subseteq\Delta$
for all $1\le s\le n$. This implies that $\Delta=\Phi$, since $\Lambda_\ell \subseteq \Lambda_s$.

\textit{Case 2:}
Suppose $\pm\epsilon_i\pm \epsilon_j\oplus y\in \Delta$ for some $i\notin J, j\in J, y\in \Lambda_\ell.$ Then since 
$\mp\epsilon_j\oplus \Lambda_s\subseteq \Delta$ and $\Lambda_\ell\subseteq \Lambda_s$, we have 
$$\pm\epsilon_i\oplus -y=(\pm\epsilon_i\pm \epsilon_j\oplus y)+(\mp\epsilon_j\oplus -2y)\in \Delta.$$
So we are back to Case 1.
\end{proof} 

When $J=\emptyset$ we have the following.
\begin{lem}\label{bcnkeyprop02}
Let $J=\emptyset$. If $\Lambda_s=\Lambda_\ell$, then the closed subroot system $\wh{A_J}$ of $\Phi$ is maximal and otherwise it
is contained in a maximal closed subroot system of the form
$$\Psi(Z)=\{\alpha \oplus Z :\a\in \mathring{\Phi}_s\}  
\cup \{2\alpha \oplus \Lambda_d, \beta\oplus \Lambda_\ell :\a\in \mathring{\Phi}_s, \beta\in \mathring{\Phi}_\ell \},$$
where $Z$ is a union of cosets modulo $\Lambda_\ell$ including $\Lambda_\ell$ and $Z$ is maximal in $\Lambda_s/\Lambda_\ell$.
\end{lem}

\begin{proof}
First assume that $\Lambda_s=\Lambda_\ell$ and $J=\emptyset$. 
Let $\Delta$ be a closed subroot system of $\Phi$ such that 
$\wh{A_J}\subsetneq\Delta\subseteq\Phi$. 
Then we have $\epsilon_i\oplus y\in \Delta\backslash\wh{A_J}$ for some $y\in \Lambda_s$. Let $j\in \{1,\dots,n\}$ such that $j\neq i$, then
we have $$(\epsilon_i\oplus y)+((-\epsilon_i+\epsilon_j)\oplus \Lambda_s)=(\epsilon_j\oplus \Lambda_s)\subseteq \Delta.$$
By interchanging the roles of $j$ and $i$, we get $(\epsilon_i\oplus \Lambda_s)\subseteq \Delta.$ This proves $\Delta=\Phi$ and 
$\wh{A_J}$ is maximal in this case.

Now assume that $\Lambda_\ell$ is properly contained in $\Lambda_s.$
Suppose $Z$ is a union cosets modulo $\Lambda_\ell$, then $Z+\Lambda_\ell=Z$ and since $\Lambda_d\subseteq \Lambda_\ell$ we get that $\Psi(Z)$ is a closed subroot system of $\Phi.$
 Clearly, $\Psi(Z)$ is maximal if and only if $Z$ is maximal in $\Lambda_s/\Lambda_\ell$ and since $\wh{A_J}\subseteq \Psi(Z)$ we are done.
\end{proof}
Let $\Psi$ be a maximal closed subroot system of $\Phi$ and set $J_\Psi=\{i\in\{1,\dots,n\} : \epsilon_i\in \mathrm{Gr}(\Psi)\}.$
\begin{cor}\label{23den}
 We have
 \begin{enumerate}
 \item $J_\Psi$ must be non--empty if $\Lambda_s\neq \Lambda_\ell.$
 \item  $\Psi=\wh{A_{J_{\Psi}}}$ if $J_\Psi$ is a proper subset of $\{1,\dots,n\}.$
 \end{enumerate}

 \end{cor}

\begin{proof}
The result is immediate from Proposition~\ref{bcnkeyprop01} and Lemma~\ref{bcnkeyprop02} if we prove that $\mathrm{Gr}(\Psi) \subseteq A_{J_{\Psi}}$. Suppose $\mathrm{Gr}(\Psi) \nsubseteq A_{J_{\Psi}}$, then there exists $k \in J_{\Psi}$ and $\ell \notin J_{\Psi}$ such that $\epsilon_k\pm\epsilon_\ell \in \mathrm{Gr}(\Psi)$. 
This means that there exists $y\in \Lambda_s, y'\in \Lambda_\ell$ such that
$\epsilon_k\oplus y\in \Psi, \epsilon_k\pm\epsilon_\ell\oplus y'\in \Psi.$
Since $\Psi$ is closed in $\Phi$, we get $$(\epsilon_k\pm\epsilon_\ell\oplus y')-(\epsilon_k\oplus y)=(\pm \epsilon_\ell\oplus (y'-y)) \in \Psi,$$
which contradicts the fact that $\ell \notin J_{\Psi}$. 
So, $\mathrm{Gr}(\Psi) \subseteq A_{J_{\Psi}}$ and hence $\Psi \subseteq \wh{A_{J_{\Psi}}}$. Now the Statements $(1)$ and $(2)$ are clear since $\wh{A_J}$ is closed in $\Phi$.
\end{proof}

\subsection{}
Recall from Lemma~\ref{inclusionen} (or our mild assumption) that $\Lambda_\ell$ is a subgroup in this case and set 
$\Gamma_\ell=\{(\a,\beta)\in\mathring{\Phi}_\ell\times \mathring{\Phi}_\ell: \alpha+\beta\in\mathring{\Phi}_d\}$. It is clear that given any $\gamma\in \mathring{\Phi}_d$, there exists
$(\a, \beta)\in \Gamma_\ell$ such that $\a+\beta=\gamma$.

\begin{lem}\label{keylemmaBCn}
 Let $\Phi$ be an irreducible affine reflection system with $\mathring{\Phi}$ of type $BC_n, n\geq 2$ and let $\Psi\subseteq \Phi$ be a closed subroot system such that $\mathring{\Phi}_s\subseteq \mathrm{Gr}(\Psi)$. Then we have
 \begin{enumerate}
  \item $\Psi\cap \Phi^{B_n}$ is a closed subroot system of $\Phi^{B_n}$. 
  \item $\Psi\cap \Phi^{B_n}=\Phi^{B_n}$ if and only if $\Psi=\Phi^{\text{re}}.$  
  
 \end{enumerate} 
\end{lem}

\begin{pf}
 If $\a, \beta\in \Psi\cap \Phi^{B_n}$ and $\a+\beta\in \Phi^{B_n}$, then we have $\a+\beta\in \Psi \cap \Phi^{B_n}$ since $\Psi$ is closed and $\Phi^{B_n}\subseteq \Phi^{\text{re}}.$
 Now assume that  $\Psi\cap \Phi^{B_n}=\Phi^{B_n}$. Then we have $Z_\a(\Psi)=\Lambda_\a$ for all $\a\in \mathring{\Phi}_s\cup \mathring{\Phi}_\ell.$
 Let $\gamma\in \mathring{\Phi}_d$ and $(\a, \beta)\in \Gamma_\ell$ such that $\a+\beta=\gamma$. This implies $Z_\gamma(\Psi)\supseteq\big(Z_\a(\Psi)+Z_\beta(\Psi) \big)\cap \Lambda_d=(\Lambda_\ell+\Lambda_\ell)\cap\Lambda_d=\Lambda_d$
since $\Lambda_d\subseteq \Lambda_\ell$ and $\Lambda_\ell$ is a group. This proves that $\gamma\oplus \Lambda_d\subseteq \Psi$ for all $\gamma \in \mathring{\Phi}_d.$ Hence we have 
$\Psi=\Phi^{\text{re}}$.\end{pf}
\subsection{}Now we are able to finish the proof of Theorem~\ref{mainBCn}. From Proposition~\ref{bcnkeyprop01} and Lemma~\ref{bcnkeyprop02} we get that $\wh{A_J}$ is a maximal closed subroot system (under the conditions of Theorem~\ref{mainBCn} (1)). The proof of the fact that the subsets in Theorem~\ref{mainBCn}(2) and Theorem~\ref{mainBCn}(3) repsectively are maximal closed subroot systems is similar to the proof given in Section~\ref{section4} and Section~\ref{section5} respectively and we omit the details. 

So in remains to show that any given maximal closed subroot system $\Psi$ has the desired form. If $J_{\Psi}$ is a proper subset we get with Corollary~\ref{23den} that $\Psi=\wh{A_{J_{\Psi}}}$ (note that we must have $J_{\Psi}\neq \emptyset$ if $\Lambda_s\neq \Lambda_\ell$ by Lemma~\ref{bcnkeyprop02}). Otherwise we have $\mathring{\Phi}_s\subseteq \mathrm{Gr}(\Psi)$ and we can assume with Lemma~\ref{keylemmaBCn} that $\Psi\cap \Phi^{B_n}$ is proper closed subroot system in
$\Phi^{B_n}.$ So we can chose a maximal closed subroot system $\Psi'$ of $\Phi^{B_n}$ such that $$\Psi\cap \Phi^{B_n}\subseteq \Psi'\subsetneq \Phi^{B_n}.$$ 

\textit{Case 1:} Assume that $\text{Gr}(\Psi')$ is semi--closed. We must have
with Theorems~\ref{Bncase} that there exists a proper non--empty subset $J\subseteq \{1,\dots,n\}$ and disjoint subsets $Z_1,Z_2\subseteq \Lambda_s$ such that $Z_1$ and $Z_2$ are unions of cosets modulo $\Lambda_\ell$, $\Lambda_s=Z_1\cup Z_2$ and
$$\Psi'=\{\pm\epsilon_i\oplus Z_1, \pm\epsilon_j\oplus Z_2: i\notin J, j\in J\}\cup\{(\pm\epsilon_p\pm\epsilon_q)\oplus \Lambda_{\ell}: p,q\notin J \text{ or } p,q\in J\}.$$
The closed subroot system spanned by $\Psi'$ in $\Phi$ is obviously contained in 
$$\Psi(Z_1, Z_2):=\Psi'\cup \{\pm2\epsilon_i\oplus \Lambda_d : 1\leq i\leq n\}.$$
Since $\Lambda_d\subseteq \Lambda_\ell$ and $\Lambda_\ell+Z_i\subseteq Z_i$ for $i=1, 2$ and $\{\pm2\epsilon_i\oplus \Lambda_d :1\leq i\leq n\}\subseteq \Psi(Z_1, Z_2)$, we immediately get that $\Psi(Z_1, Z_2)$ is a closed subroot system of $\Phi$ such that
 $$\Psi\subseteq \Psi(Z_1, Z_2)\subsetneq \Phi.$$ This implies $\Psi=\Psi(Z_1, Z_2)$ since $\Psi$ is maximal closed in $\Phi.$

\vspace{0,1cm}

\textit{Case 2:} Assume that $\text{Gr}(\Psi')$ is full. We must have by Theorem~\ref{Bnformcasefull} 
that 
$$\Psi'=\{\a\oplus \tau(\a)+Z: \a\in \mathring{\Phi}_s\}\cup\{\a\oplus \tau(\a)+H : \a\in \mathring{\Phi}_\ell\}$$ 
where $H=\Lambda_\ell$ or $H\subseteq \Lambda_\ell$ is a maximal subgroup and
$\tau:\mathring{\Phi}\to \Lambda_s$ is a $\mathbb{Z}-$linear function satisfying $\tau(\beta)\in \Lambda_\ell$ for all $\beta\in \mathring{\Phi}_\ell$ and
$Z$ is a union of cosets modulo $H$ maximal with the property that $H\subseteq Z$, $-Z=Z$ and
$$(x+H), (y+H)\subseteq Z \text{ and } (x+H)\neq (y+H)\Rightarrow (x+y) \notin \Lambda_\ell.$$
The closed subroot system spanned by $\Psi'$ in $\Phi$ is contained in 
$$\Psi(Z, H)=\Psi'\cup \{\a \oplus (\tau(\a)+H)\cap \Lambda_d : \a\in\mathring{\Phi}_d\}.$$
Since $\Lambda_d\subseteq \Lambda_\ell$ and $\Lambda_\ell+Z\subseteq Z$ and $\{\a \oplus (\tau(\a)+H)\cap \Lambda_d : \a\in\mathring{\Phi}_d\}\subseteq \Psi(Z,H)$, 
we have $\Psi(Z,H)$ is a closed subroot system of $\Phi$ such that
 $\Psi\subseteq \Psi(Z, H)$. This immediately implies that
$\Psi= \Psi(Z, H)$ since $\Psi$ is maximal closed in $\Phi.$

\vspace{0,1cm}

\textit{Case 3:} Assume that $\text{Gr}(\Psi')$ is proper and closed. But this case is impossible since $\mathring{\Phi}_s\subseteq \mathrm{Gr}(\Psi)\subseteq\mathrm{Gr}(\Psi')$ and hence $\mathring{\Phi}_\ell\subseteq \mathrm{Gr}(\Psi')$ which contradicts the fact that $\mathrm{Gr}(\Psi')$ is proper.

\section{Applications: The toroidal Lie algebra case}\label{section8}
We discuss the example of a root system of a toroidal Lie algebra.
\subsection{}Toroidal Lie algebras are $k$--variable generalizations of affine Kac--Moody Lie algebras. Recall their root system from Example~\ref{ex1} and note that the corresponding untwisted affine reflection system is given by
\begin{equation}\label{158}\Phi=\{\alpha\oplus \mathbb{Z}^k : \alpha\in \mathring{\Phi}\cup \{0\}\}.\end{equation}
Assume for simplicity that $\mathring{\Phi}$ is reduced. Let $\Psi\subseteq \Phi$ a maximal closed subroot system. If $\text{Gr}(\Psi)\subsetneq\mathring{\Phi}$ we obtain with Theorem~\ref{propgrm} that 
\begin{equation*}\Psi=\{\alpha \oplus \bz^k: \alpha\in\mathring{\Phi}'\},\end{equation*}
where $\mathring{\Phi}'$ is a maximal closed subroot system in $\mathring{\Phi}$. When the gradient is full we get from Theorem~\ref{part2g2f4} and Theorem~\ref{Bnformcasefull} there exists a maximal subgroup $H\subseteq \mathbb{Z}^k$ and $\bz$--linear function $\tau:\mathring{\Phi}\rightarrow \bz^k/H$ such that
\begin{equation*}\Psi=\{\alpha \oplus \tau(\a): \alpha\in\mathring{\Phi}\}.\end{equation*}
We will describe the maximal closed subroot systems more explicitly.
\subsection{}Let $H\subseteq \bz^k$ a maximal subgroup, then there exists a prime number $q\in\bz_+$ such that $\bz^k/H\cong \bz/q\bz$. So we can fix a $\bz$--basis $v_1,\dots,v_k$ of $\bz^k$ and an integer $1\leq i\leq k$ such that $v_1,\dots,v_{i-1},qv_i,v_{i+1},\dots,v_k$ is a $\bz$--basis of $H$, i.e.
$$H=\bz v_1\oplus \cdots \oplus \bz v_{i-1}\oplus \bz qv_i\oplus \bz v_{i+1}\oplus \dots\oplus \bz v_k.$$ If $U$ is a unimodular $(k\times k)$ square matrix, i.e. $U$ is a matrix with integer entries and determinant $\pm 1$, then the rows of $U\cdot R$ where $R^{\top}=(v_1,\dots,qv_i,\dots,v_k)$ is again a $\bz$--basis of $H$. Given a square matrix $A$, there is always a unimodular matrix such that $U\cdot A$ is of Hermite normal form (see \cite[Chapter 4]{Co93} for more details), which means the following: 
\begin{itemize}
\item $U\cdot A$ is upper triangular 
\item The first nonzero entry from the left of a nonzero row is always strictly to the right of the leading coefficient of the row above it and it is positive.
\item All zero rows are located at the bottom and the elements below pivots are zero and elements above pivots are nonnegative and strictly smaller than the pivot.
\end{itemize}
So if $R$ is as above and $U$ is unimodular such that $U\cdot R$ is of Hermite normal form, we have $\text{det}(U)\text{det}(R)=q$. 
\begin{prop}Let $q$ be a prime number and denote by $\mathbb{H}_q$ the set of $(k\times k)$ Hermite normal form matrices with determinant $q$. Further let $\mathcal{M}_q$ the set of maximal subgroups of $\bz^k$ of finite index $q$. Then there is a bijection $\mathcal{M}_q\cong \mathbb{H}_q$.
\proof
If $H\in \mathcal{M}_q$, the discussion above shows that we can always choose a $\bz$--basis of $H$ in a way that the transpose of the matrix build by the basis vectors is an element of  $\mathbb{H}_q$. So we get a map
$$\Psi: \mathcal{M}_q\rightarrow \mathbb{H}_q$$
which is clearly injective. We will show that $\Psi$ defines a bijection. Clearly $\mathbb{H}_q$ has cardinality $(q^k-1)/(q-1)$. Every element $H\in \mathcal{M}_q$ defines a surjective map 
$$\bz^k\rightarrow \bz^k/H\cong \bz/q\bz.$$ So the cardinality of $\mathcal{M}_q$ is obviously equal to the cardinality of the set 
$$\{f:\bz^k\rightarrow \bz/q\bz : f \text{ surjective homomorphism}\}/\sim$$
where we say that $f\sim g$ if and only if $\text{ker}(f)=\text{ker}(g)$. Clearly $f\sim g$ if and only if
$$\bar{f}\circ \bar{g}^{-1}\in \text{Aut}(\bz/q\bz)\cong (\bz/q\bz)^{*},$$
where $\bar{f}$ is the induced isomorphism $\bz^k/\text{ker}(f)\rightarrow \bz/q\bz$. Since we have that $\{f:\bz^k\rightarrow \bz/q\bz : f \text{ surjective homomorphism}\}$ has cardinality $(q^k-1)$ we have the desired property.
\endproof
\end{prop}
Recall that we have assumed for simplicity that $\mathring{\Phi}$ is reduced. Fix a fundamental system $\{\a_1,\dots,\a_n\}\subseteq\mathring{\Phi}$. Given $b_{\alpha_i}\in \bz$ for all $1\leq i\leq n$ and $\alpha=\sum_ia_i \a_i$ we set $b_\a=\sum a_i b_{\a_i}$.
\begin{cor}The maximal closed subroot systems with full gradient of the affine reflection system \eqref{158} are in one-to-one correspondence with 
$$\{(q,(b_{\a_i})_{},U): q \text{ is a prime number,\  $(b_{\a_i})_{}\in [0,q-1]^{n}$ and $U\in \mathbb{H}_q$}\}.$$
In particular, the maximal closed subroot system associated to a triple $(q,(b_{\a_i}),U)$ is given by
$$\{\alpha \oplus b_{\a} e_\ell+\bz u_1+\cdots +\bz u_k: \alpha\in\mathring{\Phi}\},$$
where $u_1,\dots, u_k$ are the rows of $U$ and $\ell$ is the unique integer such that $u_\ell=q e_\ell$ ($e_\ell$ is the $\ell$--th unit vector).
\hfill\qed
\end{cor}

\begin{example}Let $k=2$. Then all matrices in $\mathbb{H}_q$ are given by
$$
\begin{pmatrix}
1 & x \\
0 & q  \\
\end{pmatrix}, \  0\leq x\leq q-1,\ 
\begin{pmatrix}
q & 0 \\
0 & 1  \\
\end{pmatrix}.
$$ 
So if $\Psi$ is a maximal closed subroot system with full gradient it must be of the form
$$\Psi_{q,(b_\a),x}=\left\{\alpha\oplus \begin{pmatrix}
0 \\
b_\a \\
\end{pmatrix}+\bz \begin{pmatrix}
1 \\
x \\
\end{pmatrix}+ \bz \begin{pmatrix}
0 \\
q \\
\end{pmatrix}: \alpha\in \mathring{\Phi}\right\}$$
or 
$$\Psi_{q, (b_\a)}=\left\{\alpha\oplus \begin{pmatrix}
b_\a \\
0 \\
\end{pmatrix}+\bz \begin{pmatrix}
q \\
0 \\
\end{pmatrix}+ \bz \begin{pmatrix}
0 \\
1 \\
\end{pmatrix}: \alpha\in \mathring{\Phi}\right\}.$$
\end{example}
\subsection{} We will now discuss Saito's affine reflection system of nullity $2$ (see \cite[Page 117]{Sai85}). Let $\mathring{\Phi}$ be of type $C_n$ with fundamental system $\{\a_1,\dots,\a_n\}$ and recall that the set of short roots spans a root system of type $D_n$ with fundamental system $\{\epsilon_i-\epsilon_{i+1},\epsilon_{n-1}+\epsilon_n,\ 1\leq i\leq n-1\}$.  We consider Saito's affine reflection system
$$\Phi=\{\a\oplus(\mathbb{Z}\oplus \mathbb{Z}): \alpha\in \mathring{\Phi}_s\}\cup \{\a\oplus(\mathbb{Z}\oplus 2\mathbb{Z}): \alpha\in \mathring{\Phi}_\ell\}.$$ Recall the maximal subgroups of $\bz^2$ and the definition of $b_\a's$ from the previous subsection. By Theorems~\ref{Cng2case},~\ref{part2g2f4} and ~\ref{propgrm}, we have that $\Psi$ is a maximal closed subroot system of $\Phi$ if and only if one of the following cases hold:
\begin{enumerate}
 \item There is a $\bz$--linear function $\tau: \mathring{\Phi}_s\rightarrow \left\{\begin{pmatrix}
0 \\
0 \\
\end{pmatrix}, \begin{pmatrix}
0 \\
1 \\
\end{pmatrix}\right\}$ such that $\tau_{\epsilon_{n-1}-\epsilon_n}\neq \tau_{\epsilon_{n-1}+\epsilon_n}$ and $$\Psi=\{\a \oplus (\tau_\a+\mathbb{Z}\oplus 2\mathbb{Z}):\a\in \mathring{\Phi}_s\}.$$
\item There is a prime number $q$ and an arbitrary tuple of integers $(b_{\a_i})_{1\leq i\leq n}\in [0,q-1]^n$ such that 
\begin{align*}\Psi=\bigg\{\alpha\oplus  \begin{pmatrix}
b_\a \\
0 \\
\end{pmatrix}+\bz\begin{pmatrix}
q \\
0 \\
\end{pmatrix}& + \bz \begin{pmatrix}
0 \\
1\\
\end{pmatrix}: \alpha\in \mathring{\Phi}_s\bigg\}&\\&\cup \bigg\{\alpha\oplus \begin{pmatrix}
b_\a \\
0 \\
\end{pmatrix}+\bz\begin{pmatrix}
q \\
0 \\
\end{pmatrix}+ 2\bz \begin{pmatrix}
0 \\
1 \\
\end{pmatrix}: \alpha\in \mathring{\Phi}_\ell\bigg\}.\end{align*}
\item  There is a prime number $q$ and $x\in\{0,\dots,q-1\}$ and a tuple of integers $(b_{\a_i})_{1\leq i\leq n}\in [0,q-1]^n$ satisfying the property that $b_{\alpha_n}$ is even in the case when $q=2$ and $x=0$ (otherwise arbitrary) such that
 \begin{align*}\hspace{0,8cm}\Psi=\bigg\{\alpha\oplus \begin{pmatrix}
0\\
b_\a  \\
\end{pmatrix}+\bz\begin{pmatrix}
1 \\
x \\
\end{pmatrix}& + \bz \begin{pmatrix}
0 \\
q \\
\end{pmatrix}: \alpha\in \mathring{\Phi}_s\bigg\}&\\&\cup \bigg\{\alpha\oplus \begin{pmatrix}
0 \\
b_\a  \\
\end{pmatrix}+c\begin{pmatrix}
1 \\
x \\
\end{pmatrix}+ d \begin{pmatrix}
0 \\
q \\
\end{pmatrix}: \alpha\in \mathring{\Phi}_\ell,\ (c,d)\in Z\bigg\},\end{align*}
where $Z=\{(c,d)\in \bz^2: cx+dq\equiv b_{\alpha_n}\mod 2\}$.
 \item There is a maximal closed subroot system $\mathring{\Phi}'$ of $\mathring{\Phi}$ such that $$\Psi=\{\a \oplus (\mathbb{Z}\oplus \mathbb{Z}):\a\in \mathring{\Phi}'\cap\mathring{\Phi}_s\}\cup \{\a \oplus (\mathbb{Z}\oplus 2\mathbb{Z}):\a\in \mathring{\Phi}'\cap\mathring{\Phi}_\ell\}.$$
\end{enumerate}


\bibliographystyle{plain}
\bibliography{bibfile}

\begin{thebibliography}{10}

\bibitem{ABP14}
Bruce Allison, Stephen Berman, and Arturo Pianzola.
\newblock Multiloop algebras, iterated loop algebras and extended affine {L}ie
  algebras of nullity 2.
\newblock {\em J. Eur. Math. Soc. (JEMS)}, 16(2):327--385, 2014.

\bibitem{AABGP97}
Bruce~N. Allison, Saeid Azam, Stephen Berman, Yun Gao, and Arturo Pianzola.
\newblock Extended affine {L}ie algebras and their root systems.
\newblock {\em Mem. Amer. Math. Soc.}, 126(603):x+122, 1997.

\bibitem{BKV18}
Leon Barth, Deniz Kus, and R.~Venkatesh.
\newblock Regular subalgebras of extended affine {L}ie algebras of nullity 2.
\newblock {\em in preparation}.

\bibitem{BdS}
A.~Borel and J.~De~Siebenthal.
\newblock Les sous-groupes ferm\'es de rang maximum des groupes de {L}ie clos.
\newblock {\em Comment. Math. Helv.}, 23:200--221, 1949.

\bibitem{Bou46}
Nicolas Bourbaki.
\newblock {\em Lie groups and {L}ie algebras. {C}hapters 4--6}.
\newblock Elements of Mathematics (Berlin). Springer-Verlag, Berlin, 2002.
\newblock Translated from the 1968 French original by Andrew Pressley.

\bibitem{Carter}
R.~W. Carter.
\newblock Conjugacy classes in the {W}eyl group.
\newblock {\em Compositio Math.}, 25:1--59, 1972.

\bibitem{CNPY16}
V.~Chernousov, E.~Neher, A.~Pianzola, and U.~Yahorau.
\newblock On conjugacy of {C}artan subalgebras in extended affine {L}ie
  algebras.
\newblock {\em Adv. Math.}, 290:260--292, 2016.

\bibitem{Co93}
Henri Cohen.
\newblock {\em A course in computational algebraic number theory}, volume 138
  of {\em Graduate Texts in Mathematics}.
\newblock Springer-Verlag, Berlin, 1993.

\bibitem{DL11}
M.~J. Dyer and G.~I. Lehrer.
\newblock Reflection subgroups of finite and affine {W}eyl groups.
\newblock {\em Trans. Amer. Math. Soc.}, 363(11):5971--6005, 2011.

\bibitem{DL11a}
M.~J. Dyer and G.~I. Lehrer.
\newblock Root subsystems of loop extensions.
\newblock {\em Transform. Groups}, 16(3):767--781, 2011.

\bibitem{Dynkin}
E.~B. Dynkin.
\newblock Regular semisimple subalgebras of semisimple {L}ie algebras.
\newblock {\em Doklady Akad. Nauk SSSR (N.S.)}, 73:877--880, 1950.

\bibitem{FRT08}
Anna Felikson, Alexander Retakh, and Pavel Tumarkin.
\newblock Regular subalgebras of affine {K}ac-{M}oody algebras.
\newblock {\em J. Phys. A}, 41(36):365204, 16, 2008.

\bibitem{KB90}
Raphael H\o~egh Krohn and Bruno Torr\'esani.
\newblock Classification and construction of quasisimple {L}ie algebras.
\newblock {\em J. Funct. Anal.}, 89(1):106--136, 1990.

\bibitem{Hu80}
James~E. Humphreys.
\newblock {\em Introduction to Lie Algebras and Representation Theory},
  volume~9 of {\em Graduate Texts in Mathematics}.
\newblock Springer-Verlag, Berlin, 1980.

\bibitem{K90}
Victor~G. Kac.
\newblock {\em Infinite-dimensional {L}ie algebras}.
\newblock Cambridge University Press, Cambridge, third edition, 1990.

\bibitem{LN11}
Ottmar Loos and Erhard Neher.
\newblock Reflection systems and partial root systems.
\newblock {\em Forum Math.}, 23(2):349--411, 2011.

\bibitem{Neher04}
Erhard Neher.
\newblock Extended affine {L}ie algebras.
\newblock {\em C. R. Math. Acad. Sci. Soc. R. Can.}, 26(3):90--96, 2004.

\bibitem{N11F}
Erhard Neher.
\newblock Extended affine {L}ie algebras---an introduction to their structure
  theory.
\newblock In {\em Geometric representation theory and extended affine {L}ie
  algebras}, volume~59 of {\em Fields Inst. Commun.}, pages 107--167. Amer.
  Math. Soc., Providence, RI, 2011.

\bibitem{Ra04}
S.~Eswara Rao.
\newblock On representations of toroidal {L}ie algebras.
\newblock In {\em Functional analysis {VIII}}, volume~47 of {\em Various Publ.
  Ser. (Aarhus)}, pages 146--167. Aarhus Univ., Aarhus, 2004.

\bibitem{RV17}
Krishanu Roy and R.~Venkatesh.
\newblock Maximal closed subroot systems of affine root systems.
\newblock {\em arXiv:1707.07981}, 2017.

\bibitem{Sai85}
Kyoji Saito.
\newblock Extended affine root systems. {I}. {C}oxeter transformations.
\newblock {\em Publ. Res. Inst. Math. Sci.}, 21(1):75--179, 1985.

\bibitem{Sai90}
Kyoji Saito.
\newblock Extended affine root systems. {II}. {F}lat invariants.
\newblock {\em Publ. Res. Inst. Math. Sci.}, 26(1):15--78, 1990.

\end{thebibliography}

\end{document}